\documentclass{article}

\usepackage{graphicx}
\graphicspath{{Figures/}}

\usepackage{ifxetex}
\ifxetex
  \usepackage{fontspec}
\else
  \usepackage[T1]{fontenc}
  \usepackage[utf8]{inputenc}
  \usepackage{lmodern}
\fi
\usepackage{mathrsfs}
\usepackage{amsfonts}
\usepackage{color,soul}
\usepackage{enumerate}
\usepackage[linesnumbered,lined,boxed,commentsnumbered]{algorithm2e}
\usepackage{amssymb,latexsym,amsmath,amsthm}
\usepackage{kbordermatrix}	%rows and columns names in a matrix
\usepackage{booktabs}
\usepackage{authblk}
\usepackage{pifont}
\newcommand{\cmark}{\ding{51}}%
\newcommand{\xmark}{\ding{55}}%

\usepackage{textcomp}

\theoremstyle{plain}
\newtheorem{thm}{Theorem}
\newtheorem{lem}[thm]{Lemma}
\newtheorem{prop}[thm]{Proposition}

\newtheorem{alg}{Algorithm}
\newtheorem{rmk}[thm]{Remark}

\theoremstyle{definition}
\newtheorem{defn}{Definition}

\newtheorem{exmp}{Example}

\makeatletter
\newcommand\xleftrightarrow[2][]{%
  \ext@arrow 9999{\longleftrightarrowfill@}{#1}{#2}}
\newcommand\longleftrightarrowfill@{%
  \arrowfill@\leftarrow\relbar\rightarrow}
\makeatother

\usepackage{booktabs}
\usepackage[dvipsnames]{xcolor}
\usepackage{tikz}
\usepackage{array}

\newcolumntype{P}[1]{>{\centering\arraybackslash}p{#1}}
\newcolumntype{M}[1]{>{\centering\arraybackslash}m{#1}}

	\let\oldnl\nl% Store \nl in \oldnl
	\newcommand{\nonl}{\renewcommand{\nl}{\let\nl\oldnl}}% Remove line number for one line
\def\bfp{\mathbf{p}}
\def\OOO{\mathcal{O}}

\title{Reverse engineering of CAD models\\ via clustering and approximate implicitization}
\author{Andrea Raffo, Oliver J.D. Barrowclough and Georg Muntingh}

\begin{document}
\maketitle
In applications like computer aided design, geometric models are often represented numerically as polynomial splines or NURBS, even when they originate from primitive geometry. For purposes such as redesign and isogeometric analysis, it is of interest to extract information about the underlying geometry through reverse engineering. In this work we develop a novel method to determine these primitive shapes by combining clustering analysis with approximate implicitization. The proposed method is automatic and can recover algebraic hypersurfaces of any degree in any dimension. In exact arithmetic, the algorithm returns exact results. All the required parameters, such as the implicit degree of the patches and the number of clusters of the model, are inferred using numerical approaches in order to obtain an algorithm that requires as little manual input as possible. The effectiveness, efficiency  and robustness of the method are shown both in a theoretical analysis and in numerical examples implemented in Python.

%%%SECTION 1: INTRODUCTION
\section{Introduction}
Reverse engineering, also known as back engineering, denotes a family of techniques moving from the physical instantiation of a model to its abstraction, by extracting knowledge that can be used for the reconstruction or the enhancement of the model itself \cite{ChikofskyCross1990}. Such methods are widely used in applications to obtain CAD models, for instance from point clouds acquired by scanning an existing physical model. Among the various reasons behind this interest are the speed-up of manufacturing and analysis processes, together with the description of parts no longer manufactured or for which only real-scale prototypes are available. In mechanical part design, high accuracy models of manufactured parts are needed to model parts of larger objects that should be assembled together precisely. Industrial design and jewelry reproduction often use reverse engineered CAD models, which are usually easier to obtain than by directly designing complex free-form shapes with CAD systems. Medicine applies reverse engineering directly to the human body, such as in the creation of bone pieces for orthopedic surgery and prosthetic parts. Finally, reverse engineering is often exploited in animation to create animated sequences of pre-existing models.

In the field of isogeometric analysis (IGA), volumetric (trivariate) CAD models based on B-splines are used both for design and analysis. One major step in supporting IGA in industry is to provide backwards compatibility with today's boundary represented (B-rep) CAD models. Conversion from bivariate surface models to trivariate volume models is a challenging, and as yet unsolved problem. Methods for constructing such models include block structuring \cite{Skytt2016} and volumetric trimming \cite{Dokken2019}, or a hybrid of the two. In all cases, information about the underlying surfaces in the model is of key importance. This problem was the original motivation behind the work in this paper. However, the method can also be utilized for many of the applications outlined above in the description of reverse engineering. In particular, redesign of models is an interesting application. As CAD models are typically represented as discrete patches, modifying them individually becomes a cumbersome job when the number of patches is large. Our method clusters all patches belonging to the same primitives together. This could be a key tool in better supporting parametric design in CAD systems that are based on direct modelling.

The first step of a reverse engineering process is digitalization, i.e., the acquisition of 2-D or 3-D data point clouds. Once the acquisition is completed, model features are identified from the point clouds using segmentation techniques. Edge-based approaches to segmentation aim to find boundaries in the point data representing edges between surfaces. Faced-based segmentation proceeds in the opposite order, trying to partition a given point cloud according to the underlying primitive shapes \cite{VaradyMartinCox1997}. We refer the reader to \cite{NiuMartinSabinLangbeinBucklow2015, NiuMartinLangbeinSabin2015, CAGD4, CAGD2, CAGD3, CAGD1} for further strategies of feature detection in CA(G)D. Finally, surface modelling techniques are applied to represent points in each of the detected regions. Some of the most used representations are point clouds, meshes (e.g. polygons or triangles), boundary representations (e.g. NURBS or B-spline patches), constructive solid geometry (CSG) models, spatial partitioning models, and feature-based and constraint-based models \cite{Fudos2006}. An example of relevant problem related to the just-obtained representations is part-in-whole retrieval (PWR), where a part is given as a query model and all models containing such a query part are retrieved (see, for example, \cite{Xiao2011} and \cite{Muraleedharan2019}).

In this paper the geometry is reconstructed using implicit algebraic surfaces. According to a survey reported in \cite{Rossignac1987}, ``99 percent of mechanical parts can be modelled exactly if one combines natural quadrics with the possibility of representing fillets and blends.'' Hence algebraic surfaces of low degree naturally provide global representations and are thus well suited to our application of detecting which patches belong to the same underlying geometry. The formal problem statement that we solve is as follows: given a set of non-overlapping rational parametrized patches in $\mathbb{R}^n$, partition the patches into subsets corresponding to the underlying primitive shape they originate from.

The main contributions of this paper are:
\begin{itemize}
    \item The introduction of a novel algorithm to group patches of a given CAD model with respect to the underlying primitive geometry.  
    \item A theoretical study of the stability and the computational complexity of the method.
    \item The validation of the method on both synthetic and real data.
\end{itemize}

The paper is organized as follows. In Section \ref{Background} we introduce the main tools of our method, presenting basic definitions and existing results that will be used later on. In Section \ref{TheAlgorithm}, the core of our paper, we present the building blocks and mathematical results forming the foundation of the detection method. In Section \ref{SketchOfTheAlgorithm} we formalize these results into an explicit algorithm. In Section \ref{Theoreticalanalysis} we present a theoretical analysis of the robustness of the method. In Section~\ref{Examples} we present several examples of the application of our method, based on an implementation of the algorithm in Python that is available online \cite{CodiceMldtt}. Section~\ref{sec:Conclusion} concludes the main part of the paper, discussing encountered challenges and directions for future research. Finally, in Appendix \ref{Appendix} we report some of the proofs regarding the mathematical theory behind the algorithm.

%%%SECTION 2: BACKGROUND
\section{Background}
\label{Background}
\subsection{Clustering methods}
Clustering is a well-established unsupervised statistical learning technique, gathering a group of objects into a certain number of classes (or clusters) based on a flexible and non-parametric approach (see \cite{Hastie2001} and its references). The grouping is performed so as to ensure that objects in the same class are more similar to each other than to elements of other classes. The problem can be traced back to ancient Greece and has extensively been studied over the centuries for its various applications in medicine, natural sciences, psychology, economics and other fields. As a consequence, the literature on the topic is vast and heterogeneous with the yearly \textit{Classification Literature Automated Search Service} listing dozens of books and hundreds of papers published on the topic.

Mathematically, clustering refers to a family of methods aiming to partition a set $X=\{\mathbf{x}_1,\dots,\mathbf{x}_N\}$ such that all elements of a cluster share the same or closely related properties. In the context of statistics, any coordinate of a (data) point $\mathbf{x}_i=(x_{i1},\dots,x_{in})^T\in\mathbb{R}^n$ is the realization of a \textit{feature} or \textit{attribute}.

The grouping is performed by defining a notion of dissimilarity between elements of $X$.
\begin{defn}[Dissimilarity]
\label{defn:Dissimilarity}
A \textit{dissimilarity} on a set $X$  is a function $d:X\times X\rightarrow \mathbb{R}$ such that for all $\mathbf{x},\mathbf{y}\in X$:
\begin{enumerate}
\item $d(\mathbf{x},\mathbf{y})=d(\mathbf{y},\mathbf{x})$ \hfill(symmetry);
\item $d(\mathbf{x},\mathbf{y})\ge 0$ \hfill(positive definiteness);
\item $d(\mathbf{x},\mathbf{y})=0$ iff $\mathbf{x}=\mathbf{y}$ \hfill(identity of indiscernibles).
\end{enumerate}
\end{defn}
The concept of dissimilarity is more general than the notion of distance, where in addition to the three properties in Definition \ref{defn:Dissimilarity} we take in account also the triangular inequality:
\begin{defn}[Distance]
\label{defn:Distance}
A \textit{distance} on a set $X$  is a dissimilarity $d:X\times X\rightarrow \mathbb{R}$ such that:
\begin{enumerate}
\setcounter{enumi}{3}
\item $d(\mathbf{x},\mathbf{y})\le d(\mathbf{x},\mathbf{z})+d(\mathbf{z},\mathbf{y})$\qquad $\forall\, \mathbf{x},\mathbf{y},\mathbf{z}\in X$ \hfill(triangular inequality).
\end{enumerate}
\end{defn}
The choice of the dissimilarity (or distance) strongly depends on the problem of interest, since it allows one to determine which elements are closer to each other by giving a greater importance to certain properties compared to other unfavourable ones. From a geometrical viewpoint, the use of a dissimilarity rather than a distance leads to a generalized metric space \cite{Khamsi2015}.

\begin{exmp}
A traditional way to measure distances in a Euclidean space $\mathbb{R}^n$ is a Minkowski distance, i.e., a member of the family of metrics:
\[d_P(\mathbf{x},\mathbf{y}):=\left(\sum_{i=1}^n|x_i-y_i|^p \right)^{1/p}, \quad p\ge 1.\] 
Common choices in this family are the Manhattan, Euclidean and Chebyshev distances (respectively: $p=1,2$ and $\infty$).
\end{exmp}

Once the dissimilarity $d$ between elements has been chosen, it is necessary to define a generalization of $d$ to compare subsets of $X$. In this context generalization means that if two clusters are singletons, then their dissimilarity corresponds to the original dissimilarity between the respective elements. Clearly the generalization of $d$ is not uniquely determined and must be chosen on a case-by-case basis. Finally, an algorithm for grouping the data is designed via the defined $d$ so as to reach a faithful clustering after a finite number of steps.

Conventional clustering algorithms can be classified into two categories: hierarchical and partitional algorithms. There are two families of hierarchical algorithms:
\begin{itemize}
\item Agglomerative hierarchical algorithms start with singletons as clusters and proceed by merging the clusters that are closest step by step.
\item Divisive hierarchical algorithms work the other way around, starting with a single cluster containing all the elements and proceeding with a sequence of partitioning steps.
\end{itemize}

In this work we follow the agglomerative hierarchical approach in order to minimize the a priori knowledge that is required, consisting in general of only the number of clusters. Partitional algorithms require further information. In our method we propose an approach to infer this parameter, in order to keep the algorithm automatic. We emphasize that the result of different clustering methods highly depends on the considered dissimilarity, which is both an advantage and a disadvantage since it requires an additional choice to be made. For further details we refer the reader to \cite{GanMaWu} and references therein.

\subsection{Approximate implicitization}
\label{ApproximateImplicitization}
Implicitization is the process of computing an implicit representation of a parametric hypersurface. In algebraic geometry, traditional approaches to implicitization are based on Gr\"{o}bner bases, resultants and moving curves and surfaces \cite{Kotsireas2004}. These methods present different computational challenges, such as the presence of additional solutions and a low numerical stability. Over the last decades several alternatives have been introduced in CAGD to reach an acceptable trade-off between accuracy of representation and numerical stability.

Approximate implicitization (see \cite{BarrowcloughDokken2012, Dokken1997}) defines a family of algorithms for ``accurate'' single polynomial approximations. Approximate implicitization can be performed piecewise by dividing the model into smooth components. This approach is of interest in applications from computer graphics, where the models are rarely described by a single polynomial. To simplify our presentation, we will restrict our attention to hypersurfaces in $\mathbb{R}^n$ (varieties of codimension $1$), even though the classical theory can be applied to varieties of any codimension. In order to avoid degenerate parametric hypersurfaces, we assume their domains $\Omega\subset\mathbb{R}^{n-1}$ to be (Cartesian products of) closed intervals.

\begin{defn}[Exact implicitization of a parametric hypersurface]
Let $\mathbf{p}:\Omega\subset\mathbb{R}^{n-1}\to\mathbb{R}^n$ be a hypersurface in $\mathbb{R}^n$. An \emph{exact implicitization} of $\mathbf{p}$ is any nonzero $n$-variate polynomial $q$ such that:
\begin{equation*}
    q(\mathbf{p}(\mathbf{s}))=0, \quad \mathbf{s}\in\Omega
\end{equation*}
\end{defn}

\begin{defn}[Approximate implicitization of a parametric hypersurface]
Let $\mathbf{p}:\Omega\subset\mathbb{R}^{n-1}\to\mathbb{R}^n$ be a parametric hypersurface. An approximate implicitization of $\mathbf{p}$ within the tolerance $\varepsilon\ge 0$ is any nonzero $n$-variate polynomial $q$, for which there exists a continuous direction function $\mathbf{g}:\Omega\to\mathbb{S}^n$, with $\mathbb{S}^n\subset\mathbb{R}^n$ the unit sphere, and a continuous error function $\eta:\Omega\to(-\varepsilon,\varepsilon)$, such that:
\begin{equation*}
    q(\mathbf{p}(\mathbf{s})+\eta(\mathbf{s})g(\mathbf{s}))=0, \quad \mathbf{s}\in\Omega
\end{equation*}
\end{defn}

\begin{lem}[Dokken \cite{Dokken1997}]
\label{lem:MatricialExpression}
Let $q(\mathbf{x})=0$ define an algebraic hypersurface of degree $m$ and $\mathbf{p}=\mathbf{p}(\mathbf{s})$ be a (polynomial or) rational parametrization of (multi)degree $\mathbf{n}$ expressed in a given basis. Then it follows that the composition $q\big(\mathbf{p}(\mathbf{s})\big)$ can be expressed in a basis $\boldsymbol{\alpha}(\mathbf{s})$ for the polynomials of (multi)degree at most $m\mathbf{n}$. Explicitly,
\[ q\big(\mathbf{p}(\mathbf{s})\big)=(\mathbf{Db})^T\boldsymbol{\alpha}(\mathbf{s}),\]
where
\begin{itemize}
\item  $\mathbf{D}$ is a matrix built from products of the coordinate functions of $\mathbf{p}(\mathbf{s})$;
\item $\mathbf{b}$ is a column vector containing the unknown coefficients of $q$ with respect to a chosen basis for the polynomials of total degree at most $m$.
\end{itemize}
\end{lem}

Note that, if $\mathbf{b}\ne 0$ is in the nullspace of $\mathbf{D}$, then $q\big(\mathbf{p}(\mathbf{s})\big)=0$ and $\mathbf{b}$ contains the coefficients of an exact implicitization of $\mathbf{p}(\mathbf{s})$. If the kernel of $\mathbf{D}$ is trivial we look for an approximate implicit representation of $\mathbf{p}$ by minimizing the \emph{algebraic error} $||q\circ\mathbf{p}||_{\infty}$.

\begin{prop}[Dokken \cite{Dokken1997}]
\label{prop:UpperBoundAE}
Let $q(\mathbf{x})=0$ be an algebraic hypersurface of degree $m$ and $\mathbf{p}=\mathbf{p}(\mathbf{s})$ be a polynomial or rational parametrization of (multi)degree $\mathbf{n}$ expressed in a given basis. Then
\[\min_{\lVert \boldsymbol{b}\rVert_2=1}\max_{\mathbf{s}\in\Omega}|q\big(\mathbf{p}(\mathbf{s})\big)|\le\max_{\mathbf{s}\in\Omega}||\boldsymbol{\alpha}(\mathbf{s})||_2\sigma_{\min},\]
where $\sigma_{\min}$ is the smallest singular value of the matrix $\mathbf{D}$ defined in Lemma \ref{lem:MatricialExpression}.
\end{prop}
\begin{proof}
Applying the Cauchy-Schwarz inequality,
\[\begin{split}
\min_{\lVert \boldsymbol{b}\rVert_2=1}\max_{\mathbf{s}\in\Omega}|q\big(\mathbf{p}(\mathbf{s})\big)|&=\min_{\lVert \boldsymbol{b}\rVert_2=1}\max_{\mathbf{s}\in\Omega}|(\mathbf{Db})^T\boldsymbol{\alpha}(\mathbf{s})|\\
&\le\max_{\mathbf{s}\in\Omega}||\boldsymbol{\alpha}(\mathbf{s})||_2\min_{\lVert \boldsymbol{b}\rVert_2=1}||\mathbf{D}\mathbf{b}||_2\\
&=\max_{\mathbf{s}\in\Omega}||\boldsymbol{\alpha}(\mathbf{s})||_2\sigma_{\min},
\end{split}\]
where we used that the smallest singular value $\sigma_{\min}$ of $\mathbf{D}$ takes the form
\[ \sigma_{\min} = \min_{\lVert \boldsymbol{b}\rVert_2=1}||\mathbf{D}\mathbf{b}||_2. \qedhere\]
\end{proof}

Notice that the upper bound on the maximal algebraic error depends on the choice of the basis $\boldsymbol{\alpha}$. If the basis forms a non-negative partition of unity, then
\[\min_{\lVert \boldsymbol{b}\rVert_2=1}\max_{\mathbf{s}\in\Omega}|q\big(\mathbf{p}(\mathbf{s})\big)|\le\sigma_{\min}.\]
Denoting by $\mathbf{b}_i$ the singular vector corresponding to the singular value $\sigma_i$ of $\mathbf{D}$ and with $q_i$ the algebraic hypersurface identified by $\mathbf{b}_i$, we have
\begin{equation} \label{eqn:GoodnessApproxImplicit} |q_i\big(\mathbf{p}(\mathbf{s})\big)|\le\max_{\mathbf{s}\in\Omega}||\boldsymbol{\alpha}(\mathbf{s})||_2\sigma_i.\end{equation}
Thus the singular value $\sigma_i$ is a measure of how accurately $q_i(\mathbf{x})=0$ approximates $\mathbf{p}(\mathbf{s})$.

\subsubsection{Discrete approximate implicitization}
One of the fastest and simplest numerical implementations of approximate implicitization is based on a discrete least squares approximation of a point cloud $\mathscr{P}=\{\mathbf{p}(\mathbf{s}_i)\}_{i=1}^N$ sampled from the parametric manifold. The choice of using a point cloud sampled in parameter space is equivalent to choosing $\boldsymbol{\alpha}$ to be a Lagrange basis. This approach considers a \emph{collocation matrix} $\mathbf{D}$, expressed elementwise in terms of a basis $\{\pi_j\}_{j=1}^M$ of the set of $n$-variate polynomials of total degree at most $m$ as
\begin{equation}\label{eqn:CollocationMatrix}
D_{i,j}=\pi_j\big(\mathbf{p}(\mathbf{s}_i)\big), \qquad i=1,\dots,N,\qquad j=1,\dots,M.
\end{equation}
Analogously to Proposition \ref{prop:UpperBoundAE}, we can bound the algebraic error as
\[
\min_{\lVert \boldsymbol{b}\rVert_2=1}\max_{\mathbf{s}\in\Omega}|q(\mathbf{p}(\mathbf{s}))|\le\max_{\mathbf{s}\in\Omega}\lVert\boldsymbol{\alpha}(\mathbf{s})\rVert_2\sigma_{\min}\le\Lambda(\boldsymbol{\alpha})\sigma_{\min},
\]
where 
\begin{itemize}
\item $\Lambda(\boldsymbol{\alpha})$ is the Lebesgue constant from interpolation theory defined by $\Lambda(\boldsymbol{\alpha}):=\max_{\mathbf{s}\in\Omega}\lVert\boldsymbol{\alpha}(\mathbf{s})\rVert_1.$
\item $\sigma_{\min}$ is the smallest singular value of $\mathbf{D}$. This singular value depends on the point cloud $\mathscr{P}$, the total degree $m$ and the basis $\{\pi_j\}_{j=1}^M$. Thus, a more correct notation is $\sigma_{\text{min}}^{(m)}(\mathscr{P},\{\pi_j\}_{j=1}^M)$. For the sake of simplicity, the dependence on the point cloud $\mathscr{P}$, the total degree $m$ and the basis $\{\pi_j\}_{j=1}^M$ are omitted when not risking a misunderstanding. 
\end{itemize}

\setcounter{thm}{0}
We summarize the procedure to compute a discrete approximate implicitization as follows:
\begin{alg}
\label{alg:DAI}
Given a point cloud $\mathscr{P}:=\left\{\mathbf{p}(\mathbf{s}_1),\dots,\mathbf{p}(\mathbf{s}_N)\right\}$ sampled from a parametric hypersurface $\mathbf{p}$ and a degree $m$ for the implicit polynomial.
\begin{enumerate}
\item Construct the collocation matrix $\mathbf{D}$ by evaluating the polynomial basis $\{\pi_j\}_j$ at each point of $\mathscr{P}$ as in \eqref{eqn:CollocationMatrix}.
\item Compute the singular value decomposition $\mathbf{D}=\mathbf{U}\boldsymbol{\Sigma}\mathbf{V}^T$.
\item Select $\mathbf{b}=\mathbf{v}_{\min}^T$ the right singular vector corresponding to the smallest singular value $\sigma_{\min}^{(m)}\left(\mathscr{P}, \{\pi_j\}_{j=1}^M\right)$. 
\end{enumerate}
\end{alg}

\setcounter{thm}{2}
%%%SECTION 3: THE ALGORITHM
\section{The algorithm}
\label{TheAlgorithm}
CAD models can nowadays include millions of patches and billions of control points, making it necessary to develop algorithms that can process a huge amount of data. The novel approach we present combines flexibility, a controlled computational complexity and a high robustness to work in floating-point arithmetic. In exact arithmetic, the algorithm returns exact results. Our method does not require knowledge of the degree of the patches or the number of primitive shapes from which a certain model originates. Our idea is to infer these parameters using numerical approaches, to keep the algorithm automatic when they are not a priori available.

In this section the individual parts of the proposed reverse engineering algorithm are described in detail. We start with a pre-processing step aimed at dividing the set of all components in subsets according to their degree. Next, we describe a strategy for grouping the patches according to the primitive shapes they originate from, followed by description of how the free parameters are set. We end the section with a sketch of the complete algorithm.

\subsection{Calibration and pre-processing}
\label{CalibrationPreProcessing}
Let $X$ be the set of all the patches composing the model in $\mathbb{R}^n$ (here: $n=2,3$). The pre-processing step consists of  the partitioning $X=\cup_iX_i$ of the patches according to the suspected degree of their implicit form. The following lemma shows that we can use approximate implicitization to achieve our purpose.

\begin{lem}
\label{rmk:Degreecons}
Let $\tau$ be a non-trivial polynomial or rational patch in $\mathbb{R}^n$. Then the implicit degree of $\tau$ can be written as
\begin{equation}\label{eqn:ImplicitDegree}m:=\min\left\{\bar{m}\in\mathbb{N}^\ast \text{ s.t. } \sigma_{\min}^{(\bar{m})}=0\right\},
\end{equation}
where $\sigma_{\min}^{(\bar{m})}$ is the shorthand notation for the smallest singular value of Algorithm \ref{alg:DAI}. Here, $\mathscr{P}$ is a point cloud sampled on $\tau$ and $\{\pi_j\}_{j=1}^M$ is any basis of total degree $\bar{m}$.
\end{lem}
\begin{proof}
From Proposition \ref{prop:UpperBoundAE} it follows that
\begin{itemize}
\item Approximate implicitization of degree $\bar{m}< m$ provides the coefficients of an approximate implicit representation, together with strictly positive smallest singular value that measures the accuracy of the approximation. 
\item Approximate implicitization of degree $\bar{m}=m$ provides the coefficients of the exact implicit representation, and the smallest singular value is zero. \qedhere
\end{itemize}
\end{proof}
\begin{rmk}
\label{rmk:DegreeEstimation}
Let $N_{\text{min}}$ be the minimum number of samples guaranteeing a unique exact implicitization (e.g. $N_{\text{min}}=m^2+1$ for a non-degenerate rational parametric planar curve of implicit degree $m$). Lemma \ref{rmk:Degreecons} can be extended to consider discrete approximate implicitization if at least $N_{\text{min}}$ unique samples are considered (see \cite{BarrowcloughDokken2012} for details). 
\end{rmk}

\begin{rmk}
\label{rmk:TrimmedSurfaces}
Trimmed surfaces introduce a few additional complexities, but can be dealt with in a similar way.
Given that trimming curves are often irregular, it is in general not possible to achieve a regular sampling that conforms to the boundaries of the model.
In general we may say that if the underlying surface belongs to a certain primitive, then so will any trimmed region of that surface.
Conversely, although it is possible that the trimmed region of a surface belongs to a certain primitive while the underlying surface does not, such cases are pathological.
We can thus deal with trimmed surfaces in two ways.
The first approach is to simply sample the underlying surface regularly, disregarding the trimming region. 
Another approach is to perform random oversampling of points within the trimmed region to ensure that enough samples are taken. 
The latter approach is also robust to pathological examples. 
\end{rmk}

Equation \eqref{eqn:ImplicitDegree} is useful to determine the implicit degree of a patch. This characterization for the implicit degree fails when we work in floating-point arithmetic, due to the approximations adopted in the computation. A possible modification is the relaxation of (\ref{eqn:ImplicitDegree}) by the weaker criterion
\begin{equation}\label{eqn:WeakerImplicitDegree}
m:=\min\left\{\bar{m}\in\mathbb{N}^\ast \text{ s.t. } \sigma_{\min}^{(\bar{m})}<\xi^{(\bar{m})}\right\},
\end{equation} 
where the threshold $\xi^{(\bar{m})}$ is introduced to take into account the rounding-off error of floating point arithmetic and other possible sources of uncertainty. Note that this parameter depends on the implicit degree chosen in the computation of discrete approximate implicitization.

Suppose that we want to distinguish the patches of degree $m$ from those of greater degree. A statistical approach to infer such thresholds is the following. 

\setcounter{thm}{1}
\begin{alg}
\label{alg:DegreePartitioning}
Given the implicit degree of interest $m$.
\begin{enumerate}
\item Generate:
\begin{itemize}
\item $Q_1\gg 1$ random patches of implicit degree $m$; 
\item $Q_2\gg 1$ random patches of implicit degree $m+1$.
\end{itemize} 
\item For each of these two sets, average the smallest singular values computed by applying discrete approximate implicitization of degree $m$ on each patch. Denote these two values with $\overline{\xi}_1$ and $\overline{\xi}_2$. 
\item Finally, compute $\xi^{(m)}$ as the geometric mean of $\overline{\xi}_1$ and $\overline{\xi}_2$, i.e.,
\[\xi^{(m)} = \sqrt{\overline{\xi}_1\overline{\xi}_2}.\]
\end{enumerate}
\end{alg}
The use of a geometric mean allows computation of a value whose exponent in the scientific form is intermediate to the ones of $\overline{\xi}_1$ and $\overline{\xi}_2$. Notice that in applications such as CAD, curves and surfaces have often low implicit degree (e.g. degree 1 and 2 for natural quadrics). As we will see in Section \ref{Theoreticalanalysis}, the estimation of the implicit degree can fail in floating-point arithmetic.

\setcounter{thm}{4}
\subsection{An agglomerative approach}
\label{AgglomerativeApproach}
As previously mentioned, hierarchical algorithms are subdivided into agglomerative and divisive hierarchical algorithms.
The agglomerative approach starts with each object belonging to a separate cluster. Then a sequence of irreversible steps is taken to construct a hierarchy of clusters.

In our case, $X$ is the set of patches to cluster. We want to reach a partition where patches in the same cluster are the only ones originating from the same primitive shape. We propose to derive the dissimilarity measure from discrete approximate implicitization. Since we have already partitioned $X=\cup_i X_i$ according to the suspected implicit degree of the patches, we can assume that all the patches in $X$ have same degree $m$. Starting from each patch in a single cluster, at each step the two clusters with smallest dissimilarity are merged.

As a first step to define a dissimilarity on $X$, we identify each patch with a point cloud of $N$ points sampled on its parametrization, where $N\ge N_{\text{min}}$ and $N_{\text{min}}$ is the constant of Remark \ref{rmk:DegreeEstimation}. Typically a uniform sampling scheme is chosen. This assumption will guarantee the minimum number of samples for a unique exact implicitization when considering a pair of patches lying on the same primitive shape. Let $\ddot{X}$ be the set of such point clouds, following the indexing and partitioning of $X$. Notice that the points can be chosen such that the patches in $X$ are in 1-1 correspondence to the point clouds of $\ddot{X}$, i.e.,
\[X\ni\tau \xleftrightarrow{1-1} \ddot{\tau}:=\{P_1^{\tau},\dots,P_N^{\tau}\}\in\ddot{X}.\]

\begin{defn}[Family of candidate dissimilarities $d_\lambda$]
\label{defn:Candidated}
Let $\lambda\ge 0$. For each pair of patches $\tau_1$ and $\tau_2$ in $X$, let $\ddot{\tau}_1$ and $\ddot{\tau}_2$ be their respective point clouds in $\ddot{X}$.
We define \[d_\lambda(\tau_1,\tau_2):=\sigma_{\min}^{(m)}\left(\ddot{\tau}_1\cup\ddot{\tau}_2\right)+\lambda||\mathbf{r}_{\text{CM}}(\ddot{\tau}_1)-\mathbf{r}_{\text{CM}}(\ddot{\tau}_2)||_2,\]
where 
\begin{itemize}
\item $\lambda$ is a regularization parameter.
\item $\sigma_{\min}^{(m)}\left(\ddot{\tau}_1\cup\ddot{\tau}_2\right)$ is the smallest singular value computed by applying discrete approximate implicitization of degree $m$ to the point cloud $\ddot{\tau}_1\cup\ddot{\tau}_2$.
\item $||\mathbf{r}_{\text{CM}}(\ddot{\tau}_1)-\mathbf{r}_{\text{CM}}(\ddot{\tau}_2)||_2$ is the Euclidean distance between the center of masses of the two point clouds.
\end{itemize}
\end{defn}

\begin{lem}
\label{lem:ProofDiss}
Let $X$ be a set of patches such that the centers of masses of the elements of $X$ are distinct. Let $d_\lambda$ be the map defined in Definition \ref{defn:Candidated}. Then:
\begin{enumerate}[i)]
\item The map $d_\lambda$ is a dissimilarity $\iff$ $\lambda>0$. 
\item There exists $\lambda^\ast>0$ such that  $d_\lambda$ is a distance $\iff$ $\lambda\ge\lambda^\ast$.
\end{enumerate}
\end{lem}
\begin{proof}\begin{enumerate}[i)]
\item It is obvious that $d_\lambda$ is non-negative and symmetric for any $\lambda\ge0$. Let's then prove the identity of indiscernibles. Let $\tau_1,\tau_2$ be a pair of patches in $X$. Then 
\[\begin{split}
d_\lambda(\tau_1,\tau_2)=0 &\iff \sigma_{\min}^{(m)}\left(\ddot{\tau}_1\cup\ddot{\tau}_2\right)+\lambda||\mathbf{r}_{\text{CM}}(\ddot{\tau}_1)-\mathbf{r}_{\text{CM}}(\ddot{\tau}_2)||_2=0\\ &\iff \begin{cases} \sigma_{\min}^{(m)}\left(\ddot{\tau}_1\cup\ddot{\tau}_2\right)=0 \\ \lambda||\mathbf{r}_{\text{CM}}(\ddot{\tau}_1)-\mathbf{r}_{\text{CM}}(\ddot{\tau}_2)||_2=0 \end{cases},
\end{split}\]
where the last equivalence arises from the non-negativity of the two terms. Notice that:
\begin{itemize}
\item The smallest singular value $\sigma$ is zero iff $\tau_1$ and $\tau_2$ lie on the same hypersurface of degree $m$.
\item $\lambda||\mathbf{r}_{\text{CM}}(\ddot{\tau}_1)-\mathbf{r}_{\text{CM}}(\ddot{\tau}_2)||_2=0$ iff the patches have the same center of mass or $\lambda=0$. Since the first case is excluded by hypothesis, the lemma is proved. 
\end{itemize}
\item Let $\tau_1$, $\tau_2$, $\tau_3$ be three patches in $X$. Then $d(\tau_1,\tau_2)\le d(\tau_1,\tau_3)+d(\tau_1,\tau_2)$ iff  
\[\lambda\ge\dfrac{\sigma_{\min}^{(m)}\left(\ddot{\tau}_1\cup\ddot{\tau}_2\right)-\sigma_{\min}^{(m)}\left(\ddot{\tau}_1\cup\ddot{\tau}_3\right)-\sigma_{\min}^{(m)}\left(\ddot{\tau}_2\cup\ddot{\tau}_3\right)}{\sum_{i=1}^2||\mathbf{r}_{\text{CM}}(\ddot{\tau}_i)-\mathbf{r}_{\text{CM}}(\ddot{\tau}_3)||_2-||\mathbf{r}_{\text{CM}}(\ddot{\tau}_1)-\mathbf{r}_{\text{CM}}(\ddot{\tau}_2)||_2}=:\lambda_{\tau_1,\tau_2,\tau_3},\]
which is well-defined by the triangle inequality for the Euclidean distance and since distinct patches have distinct center of masses by assumption. Hence, it follows that the triangular inequality holds (only) for $d_\lambda$ with 
\[\lambda \geq \lambda^\ast:=\max_{\tau_1,\tau_2,\tau_3\in X}\lambda_{\tau_1,\tau_2,\tau_3}. \qedhere\]
\end{enumerate}
\end{proof}

\begin{rmk} Notice that:
\begin{enumerate}[i)] 
\item Any pair of patches $\tau_1$ and $\tau_2$ belongs to the same primitive shape iff $\sigma_{\min}^{(m)}\left(\ddot{\tau}_1\cup\ddot{\tau}_2\right)=0$ or, equivalently,  iff $d_0(\tau_1,\tau_2)=0$.
\item A parameter $\lambda>0$ is required to exploit the theory of clustering analysis. On the other hand, $\lambda$ should penalize the term $||\mathbf{r}_{\text{CM}}(\ddot{\tau}_1)-\mathbf{r}_{\text{CM}}(\ddot{\tau}_2)||_2$ in order to preserve the behavior of $d_0$ described i). We can choose $\lambda\approx 0$, e.g. $\lambda=10^{-10}$, to approximate $d_0$ by a dissimilarity. We remark that this constant is typically much smaller than $\lambda^\ast$.
\item The assumption in Lemma \ref{lem:ProofDiss} is reasonable for applications such as CAD, where patches do not overlap.
\end{enumerate}
\end{rmk}

The dissimilarities between pairs of patches are stored in the $|X|\times|X|$ matrix $\mathbf{D}_X$, known as the \emph{dissimilarity matrix}.

Now that we have defined a dissimilarity between elements, we extend it to clusters by defining the map $D_{\lambda}:2^X\times2^X\to[0,+\infty)$ using a complete-linkage approach:
\[D_{\lambda}(C_i,C_j):=\max_{\tau_k\in C_i,\tau_l\in C_j}d_{\lambda}(\tau_k,\tau_l).\]
The main reason for choosing complete-linkage is its relatively low computational cost, since the dissimilarity matrix at step $k$ is just a submatrix of $\mathbf{D}_X$. In addition numerical results show that it works well in practice.

Finally, starting from each patch in a separate cluster, we merge at each step the pair of clusters with smallest dissimilarity $D_{\lambda}$. The merging continues until the correct number of primitive shapes is detected, using the stopping criterion described in the next section.

\subsection{Stopping criterion for the agglomerative approach}
\label{StoppingCriterion}
How should one estimate the final number of clusters?  We propose to define a stopping criterion considering the map $d_0$ as follows:
\begin{itemize}
\item At each iteration $k$ we compute, for each cluster, the maximum value of $d_{0}$ for pairs of patches. From a numerical viewpoint, it corresponds to an empirical estimation of the maximum error in the approximate implicitization of the patches.
\item We consider then the maximum of the maxima and denote it by $e^{(k)}$. We will refer to $e^{(k)}$ as the \emph{representation error at iteration} $k$.
\end{itemize}
Suppose $L$ is the number of underlying primitives and let $P:=|X|$ be the number of patches. Then:
\begin{itemize}
\item In exact arithmetic, $e^{(k)}=0$ for $k=1,\dots,P-L$ and $e^{(k)}>0$ for $k=P-L+1,\dots,N-1$. Therefore the number of iterations can be defined as the maximum index $k$ such that $e^{(k)}=0$, i.e.,
\begin{equation}
\label{equation:number_iterations_exact_arithmetics}
    \bar{k}:=\max\{k|e^{(k)}=0\}.
\end{equation}

\item In floating-point arithmetic, Equation \ref{equation:number_iterations_exact_arithmetics} does not, in general, return any integer. We thus aim at designing a stopping criterion that can handle round-off error. We propose to estimate the number of iterations $\bar{k}$ as the index where the representation error jumps significantly for the first time, i.e.,
\[\bar{k}:=\min\{k|e^{(k)}>\eta\}.\]
\end{itemize}

One way to choose the stopping tolerance $\eta$ is to proceed similarly to Algorithm~\ref{alg:DegreePartitioning}:

\setcounter{thm}{2}
\begin{alg}
\label{alg:StoppCriteta}
Given the maximum implicit degree of interest $m_{\max}$:
\begin{enumerate}
\item Generate $P_3\gg1$ sets $\{\mathscr{D}_i\}_{i=1}^M$ of random patches of implicit degree between $1$ and $m_{\max}$.
\item Compute the number of iterations $\bar{k}_i$ for an exact clustering of each $\mathscr{D}_i$, considering that the number of clusters for training sets is known.
\item Compute, for each $\mathscr{D}_i$, the respective representation error $e_i$ at iteration $\bar{k}_i$; 
\item Define $\eta:=\left(\min{e_i}\right)^2$.
\end{enumerate}
\end{alg}

% \begin{equation}
% \label{equation:stopping_criterion_1}
%     \bar{k}:=\left(\arg\max\limits_{k=1,\dots,N-1} e^{(k)}/e^{(k-1)}\right)-1
% \end{equation}
% or
% \begin{equation}
% \label{equation:stopping_criterion_2}
%     \bar{k}:=\min\limits_{k=1,\dots,N-1}\left\{k|e^{(k)}/e^{(k-1)}>\eta\right\}-1,
% \end{equation}
% with $e^{(0)}:=e^{(1)}$ and where $\eta$ can be defined via a learning approach -- by generating set of random patches in a similar way of Algorithm \ref{alg:DegreePartitioning} -- or set by user. Numerical simulations by these two criteria are proposed in Section \ref{Examples}.

Such an empirical threshold $\eta$ lies between $0$ and the smallest representation error of an incorrect representation. Whenever $e^{(k)}$ is smaller than $\eta$, the algorithm will proceed by joining the clusters having this smallest dissimilarity. Otherwise, the algorithm will stop.

As an alternative to this stopping criterion based on absolutes, one may prefer the following relative stopping criterion.
\begin{alg}
\label{alg:StoppCriteta2}
Given a sequence $e^{(k)}$, $k=1,\dots,P-1$, of representation errors:
\begin{enumerate}
\item Set $e^{(0)}:=e^{(1)}$.
\item Define $\tilde{e}^{(k)}:=e^{(k)}/e^{(k-1)}$, for $k=1,\dots,N-1$.
\item Define $\bar{k}:=\left(\arg\max_k\tilde{e}^{(k)}\right)-1$.
\end{enumerate}
\end{alg}

Additional details on the use of these two criteria are provided in Section \ref{Examples}.

\newpage
%%%SECTION 4: SKETCH OF THE ALGORITHM
%%%SECTION 4: SKETCH OF THE ALGORITHM
\setcounter{algocf}{4}
\section{Sketch of the algorithm for the detection of primitive shapes}
\label{SketchOfTheAlgorithm}
We now present a sketch of the complete detection algorithm described earlier in this section. The input consists of a finite set of patches $X$ lying on different manifolds and having different centers of mass. As output, the algorithm returns the partition of $X$ corresponding to the manifolds the patches come from. \\

First, we partition $X$ according to the implicit degree of the patches as described in Section \ref{CalibrationPreProcessing}. The tuning parameters $\{\xi^{(m)}\}$ are computed using Algorithm \ref{alg:DegreePartitioning}. The maximum implicit degree of the patches $m_{\max}$ is returned as a result of this first step. 

\begin{algorithm}
\DontPrintSemicolon % Some LaTeX compilers require you to use \dontprintsemicolon instead
\KwIn{The set of polynomial patches $X$.}
\KwOut{The partition $X_1,\dots,X_{m_{\max}}$ according to the implicit degree.}
Set $m:=0$ (degree for discrete approximate implicitization);\;
\While{$|X|>0$}{
	Set $m:=m+1$ (increment the degree for approximate implicitization);\;
	Compute $\xi^{(m)}$ as defined in Algorithm \ref{alg:DegreePartitioning};\;
 	Set $X_m:=\emptyset$;\;
   	\For{$\tau\in X$} {
      		\If{$\sigma_{\min}^{(m)}(\tau)<\xi^{(m)}$} {
         		Set $X_m:=X_m\cup\{\tau\}$;\;
         		Set $X:=X\backslash\{\tau\}$;\;
        	}
   	}
}
Set $m_{\max}:=m$;\;
\Return{$X_1,\dots,X_{m_{\max}}$}\;
\caption{Partition of the set of patches $X$ according to their implicit degrees.}
\label{algo:X_partitioning}
\end{algorithm}

Finally, we apply the agglomerative clustering approach described in Section \ref{AgglomerativeApproach} to each of the subsets $X_i$. We will denote with $X_i=\cup_j X_{i,j}^{(k)}$ the partition of the cluster $X_i$ into the clusters $X_{i,j}^{(k)}$ at step $k$, and with $X_i=\cup_j X_{i,j}$ the final partition corresponding to the primitive shapes. 

The tolerance $\eta$ is assumed to be previously computed as in Section \ref{StoppingCriterion}. We recall that in the agglomerative approach each patch constitutes a single cluster at the initial step. 

\begin{algorithm}
\DontPrintSemicolon % Some LaTeX compilers require you to use \dontprintsemicolon instead
\KwIn{$\left|\begin{array}{ll} X=\cup_iX_i & \mbox{the initial partition according to the implicit degree.} \\ \eta & \mbox{the threshold for the stopping criterion.}
\end{array}\right.$}
\KwOut{The final partition $X=\cup_{i,j}X_{i,j}$ according to the underlying primitive shapes.}
\For{$i=1:m_{\max}$}{
	Set $P_i:=|X_i|$ (number of clusters at step 0);\;
	\If{$P_i>1$}{
	\nonl \textbf{\underline{Initial step:}}\;
	Set $\mathbf{D}_X:=\mathbf{0}_{P_i\times P_i}$ (dissimilarity matrix at step 0);\;
	\For{$j_1=1:P_i$}{
		\For{$j_2=j_1+1:P_i$}{
			Set $\mathbf{D}_X(j_1,j_2):=d_{\lambda}\left(\tau_{j_1},\tau_{j_2}\right)$.
		}
	}
	$\mathbf{D}_X:=\mathbf{D}_X+\mathbf{D}_X^T$;\;
	Define the partition $X_i=\cup_{j=1}^{P_i}X_{i,j}^{(0)}$, where each cluster $X_{i,j}^{(0)}$ contains one and only one patch (agglomerative approach);\;
	
	\nonl \textbf{\underline{Agglomerative process:}}\;
	Set $k:=0$ and $e^{(0)}:=0$;\;
	\While{$e^{(k)}<\eta$}{
		Set $P_i^k:=|X_i|$ (number of clusters at step $k$);\;
		\uIf{$k=0$}{
			Set $\mathbf{D}_X^0:=\mathbf{D}_X$ (dissimilarity matrix at step $0$);\;
		}
		\uElse{
			Set $\mathbf{D}_X^k:= \mathbf{0}_{P_i^k\times P_i^k}$ (dissimilarity matrix at step $k$);\;
			\For{$j_1=1:P_i^k$}{
				\For{$j_2=j_1+1:P_i^k$}{
					Set $\mathbf{D}_X^k(j_1,j_2):=D_{\lambda}(X_{i,j_1}^{(k)},X_{i,j_2}^{(k)})$;\;
				}
			}
		$\mathbf{D}_X^k:=\mathbf{D}_X^k+\left(\mathbf{D}_X^k\right)^T$;\;
		}	
	Find the pair of clusters $X_{i,\bar{j_1}}^{(k)}$ and $X_{i,\bar{j_2}}^{(k)}$ minimizing $D_{\lambda}$;\;
	Set the partition for step $k+1$ by merging $X_{i,\bar{j_1}}^{(k)}$ and $X_{i,\bar{j_2}}^{(k)}$ ;\;
	Compute $e^{(k+1)}$ as described in Section \ref{StoppingCriterion};\;
	Set $k:=k+1$;\;
	}
	}
}
\Return{$X=\cup X_{i,j}^{(k-1)}$}\;
\caption{Partition of the set of patches $X$ corresponding to the underlying primitive shapes.}
\label{algo:X_partitioning_primitive_shape}
\end{algorithm}

\setcounter{thm}{6}

\subsection{Computational complexity}
\label{computational_complexity}
In this section we analyze the computational complexity of the proposed algorithm in the case of curves and using approximate implicitization with total degree at most $m$.

\subsubsection{Complexity of approximate implicitization}
We start by determining the computational complexity for approximate implicitization. For curves in $\mathbb{R}^2$, the collocation matrix $\mathbf{D}$ \eqref{eqn:CollocationMatrix} has the form $(\pi_j(\bfp_i))_{i=1,j=1}^{N,M}$, involving:
\begin{itemize}
\item $N$ points $\bfp_1,\ldots, \bfp_N$ sampled from one or more segments.
\item $M$ polynomials $\pi_1,\ldots,\pi_M$ spanning the space $\mathbb{R}_m[x,y]$ of bivariate polynomials of total degree at most $m$. In particular, 
\begin{equation}\label{eqn:n} M:=\dim \mathbb{R}_m[x,y]=\dfrac{(m+1)(m+2)}{2}.\end{equation}
\end{itemize}
The computational complexity of approximate implicitization is the result of two contributions:
\begin{itemize}
\item % Since the complexity for polynomial evaluation (of total degree $m$) is $\OOO(m^2)$, the assembly of the collocation matrix costs $\OOO(MNm^2)=\OOO(Nm^4)$. 
For the monomial basis considered in this paper, using a triangular scheme to assemble the entries of $\mathbf{D}$ reduces the computational complexity to $\OOO(Nm^2)$.
\item The singular value decomposition of an $N\times M$ matrix has computational complexity $\OOO(\min(NM^2, N^2M))$ \cite{BENES20188}. In our setting, this yields
\begin{equation*}
    \OOO(\min(NM^2, N^2M))
    % = \OOO\left(N^2\dfrac{(m+1)(m+2)}{2}\right)
    = \OOO(\min(Nm^4, N^2m^2))
    % = \OOO(N^2m^2).
\end{equation*}
\end{itemize}

In case of points lying exactly on parametric curves, we can (without loss of generality) set $N:=N_{\text{min}}$, where $N_{\text{min}}$ is the minimum number of samples that guarantees a unique exact implicitization. By assuming the parametric planar curve to be rational non-degenerate, $N_{\text{min}}=m^2+1$ (see Remark \ref{rmk:DegreeEstimation}), which leads to a total computational cost of $\OOO(m^6)$. Although approximate implicitization exhibits high complexity with respect to patch degree, we are mainly concerned with low degree patches in this work, as motivated in the introduction. Due to this, the total computational time is dominated by the number of patches $P=\sum_{m=1}^{m_{\max}}P_m,$ to be clustered.

\subsubsection{Estimating the degree of the patches}
Algorithm \ref{algo:X_partitioning} runs approximate implicitization of degree $m$ a total of ${P-\sum_{k=0}^{m-1}P_k}$ times, where we set $P_0:=0$ and where $m=1,\ldots,m_{\max}$. Thus the routine for estimation of degree has $\OOO(P)$ complexity in the number of patches. This result is independent of the dimension of the hypersurfaces to be clustered.

\subsubsection{Assembling the dissimilarity matrix}
Considering the degree of the patches to be constant (e.g. $m_{\max}$ in the worst case), the complexity of assembling the dissimilarity matrix is proportional to the number of elements in the matrix, that is $\OOO(P^2),$ regardless of the dimension of the hypersurfaces to be clustered. It may be noted, moreover, that each element of the dissimilarity matrix can be computed independently, implying that the matrix assembly is highly parallelizable (\emph{embarrassingly parallel}).   

\subsubsection{Clustering procedure}

The naive implementation of agglomerative hierarchical clustering requires $\OOO(P^3)$ operations \cite{Day.Edelsbrunner1984}. However, even with this implementation, for all cases we have encountered so far, the assembly of the dissimilarity matrix requires most of the computational time. The dominance of the assembly procedure is even more prominent for surfaces in $\mathbb{R}^3.$ Thus, in practice, the number of patches required before the clustering step becomes dominant is excessively high. If the clustering of such a large number of patches is required, it is also possible to exploit the complete-linkage approach to implement the clustering procedure with $\OOO(P^2)$ complexity \cite{Defays1977}.

%Let us assume that the algorithm is run in exact arithmetic so that exact results are produced (i.e. no misclassifications). Then, it follows that:
%\begin{description}
    %\item [Algorithm \ref{algo:X_partitioning}] Let $P_0:=0$. Then, approximate implicitization of degree $m$ is run $P-\sum_{k=0}^{m-1}P_k$ times.
    %\item [Algorithm \ref{algo:X_partitioning_primitive_shape}] For a degree $m$, one needs to compute a $P_m\times P_m$ collocation matrix. Being a collocation matrix symmetric with all diagonal entries equal to zero, this implies that approximate implicitization of degree $m$ is applied for a total of $P_m(P_m-1)/2$ times. 
%\end{description}

% \begin{itemize}
%     \item In Algorithm \ref{algo:X_partitioning}, testing whether a patch has implicit degree $m$ has a computational complexity of $\OOO(m^6)$. We apply this check for a total of $P+\sum_{m=2}^{m_{\text{max}}}(P-\sum_{k=1}^{m-1}P_k)$ times, with the computational complexity of approximate implicitization increasing as the degree increases. Specifically, the total computational complexity of Algorithm \ref{algo:X_partitioning} is:
%     \begin{equation*}
%         \text{O}\left(P(N+N^2)+\sum_{m=2}^{m_{\text{max}}}\left(P-\sum_{k=1}^{m-1}P_k\right)(Nm^4+N^2m^2)\right)
%     \end{equation*}
%     \item In Algorithm \ref{algo:X_partitioning_primitive_shape}, filling up the dissimilarity matrices costs 
%     \begin{equation*}
%         \text{O}\left(\sum_{m=1}^{m_{\text{max}}}\dfrac{P_m(P_m-1)}{2}(Nm^4+N^2m^2)\right).
%     \end{equation*}
%\end{itemize}

%%%SECTION 5: THEORETICAL ANALYSIS
\section{Theoretical analysis}
\label{Theoreticalanalysis}
%Given a polynomially or rationally parametrized curve segment in %$\mathbb{R}^2$, we can uniformly sample a point cloud and compute the collocation matrix \[\mathbf{D}=\pi_k(\mathbf{p}(t_j)), \quad j=1,\dots,N \text{ and } k=1,\dots,M,\] 
%for discrete approximate implicitization of general degree $m$, where:
%\begin{itemize}
%\item $M=(m+1)(m+2)/2$ is the number of monomials in an $m$-degree bivariate polynomial.
%\item $N\ge m^2+1$, is at least the minimum number of samples %guaranteeing a unique exact implicitization (see Remark %\ref{rmk:DegreeEstimation}).
%\end{itemize}

In this section we state three propositions on the stability and robustness of discrete approximate implicitization under scaling, translation and rotation when the monomial basis is considered. The proofs are given in the appendix. These results of the study are generalizable to hypersurfaces of $\mathbb{R}^n$.

\begin{prop}[Scaling and smallest singular value]
\label{prop:scaling}
Let \[\mathbf{p}_{\mathbf{a}}(t)=\big(a_1x(t),a_2y(t)\big), \qquad t\in [a,b]\subset\mathbb{R},\] 
be a family of scaled polynomial or rational parametric segments, where $\mathbf{a} = (a_1,a_2)\in \mathbb{R}^2$ with $a_1,a_2>0$. Let $\mathscr{P}_{\mathbf{a}}:=\{\mathbf{p}_{\mathbf{a}}(t_j)\}_j$ be a family of uniformly sampled point clouds. Then
\[\lim_{(a_1,a_2)\to (0,0)}\sigma_{\min}^{(m)}\left(\mathscr{P}_{\mathbf{a}}\right)=0.\]
\end{prop}

\begin{prop}[Translations and smallest singular value]
\label{prop:translation}
Let \[\mathbf{p}_{\mathbf{a}}(t)=\big(x(t)+a_1,y(t)+a_2\big), \qquad t\in [a,b]\subset\mathbb{R},\] 
be a family of translated polynomial or rational parametric segments, where $\mathbf{a}:=(a_1, a_2)\in\mathbb{R}^2$. Let $\mathscr{P}_{\mathbf{a}}:=\{\mathbf{p}_{\mathbf{a}}(t_j)\}_j$ be a family of uniformly sampled point clouds. Then
\[\lim_{a_k\to \infty}\sigma_{\min}^{(m)}\left(\mathscr{P}_{\mathbf{a}}\right)=0, \qquad k=1,2. \]
\end{prop}

\begin{prop}[Rotation and smallest singular value]
\label{prop:rotation}
Let 
\[
\mathbf{p}_{\theta}(t)=
\begin{pmatrix}x(t) & y(t)\end{pmatrix}
\begin{pmatrix}\cos{\theta} & -\sin{\theta}\\  \sin{\theta} & \cos{\theta} \end{pmatrix}, \qquad t\in [a,b]\subset\mathbb{R},
\] 
be a family of polynomial or rational parametric segments rotated about the origin, where $\theta\in[0,2\pi]$. Let $\mathscr{P}_{\theta}:=\{\mathbf{p}_{\theta}(t_j)\}_j$ be a family of uniformly sampled point clouds. Then there exists $\alpha,$ $\beta\in\mathbb{R}_{>0}$ such that  
\[\alpha\le\sigma_{\min}^{(m)}\left(\mathscr{P}_{\theta}\right)\le\beta.\]
\end{prop}

These propositions imply that the algorithm is not guaranteed correct in floating-point arithmetic, since the smallest singular value can be small enough that the weak condition (\ref{eqn:WeakerImplicitDegree}) or the stopping criterion of Section \ref{StoppingCriterion} fail. However, these issues can be mitigated by performing further preprocessing on the data (rescaling to a fixed size, calibrating, etc.).  The next section shows that the algorithm works well in practice.

%%%SECTION 6: EXAMPLES
\section{Experimental results}
\label{Examples}
Our algorithm has been tested on two- and three-dimensional examples. We first describe its behavior for segments originating from lines and conics in a plane, then we study its correctness, robustness and efficiency when applied to line segments and B\'ezier approximations of circular segments, and finally we test its correctness on a real-world industrial example of a 3D CAD model.

\subsection{Straight lines and conics}
Let $\mathscr{F}_c=\{\mathscr{C}_i\}_{i=1}^L$ be a family of $L$ curves, where the type of each curve (e.g. straight line, parabola, ellipse, hyperbola) is randomly chosen with equal probability. Each curve $\mathscr{C}_i$ is obtained as follows: random parameters are sampled to define a rotation about the origin, a dilation and a translation of the canonical curve chosen as representative of the type of $\mathscr{C}_i$. A family of $P_i$ continuous subsegments is sampled. The extracted segments are gathered in the set $X$. The method is tested on the set $X$, which, after a suitable rescaling, is contained in the region $[-1,1]\times[-1,1]$. An example is shown in Figure \ref{fig:Ex1}.

\begin{figure}[htp]
\begin{minipage}[b]{6.5cm}
\centering
\includegraphics[scale=0.36, clip = true, trim = 30 20 0 40]{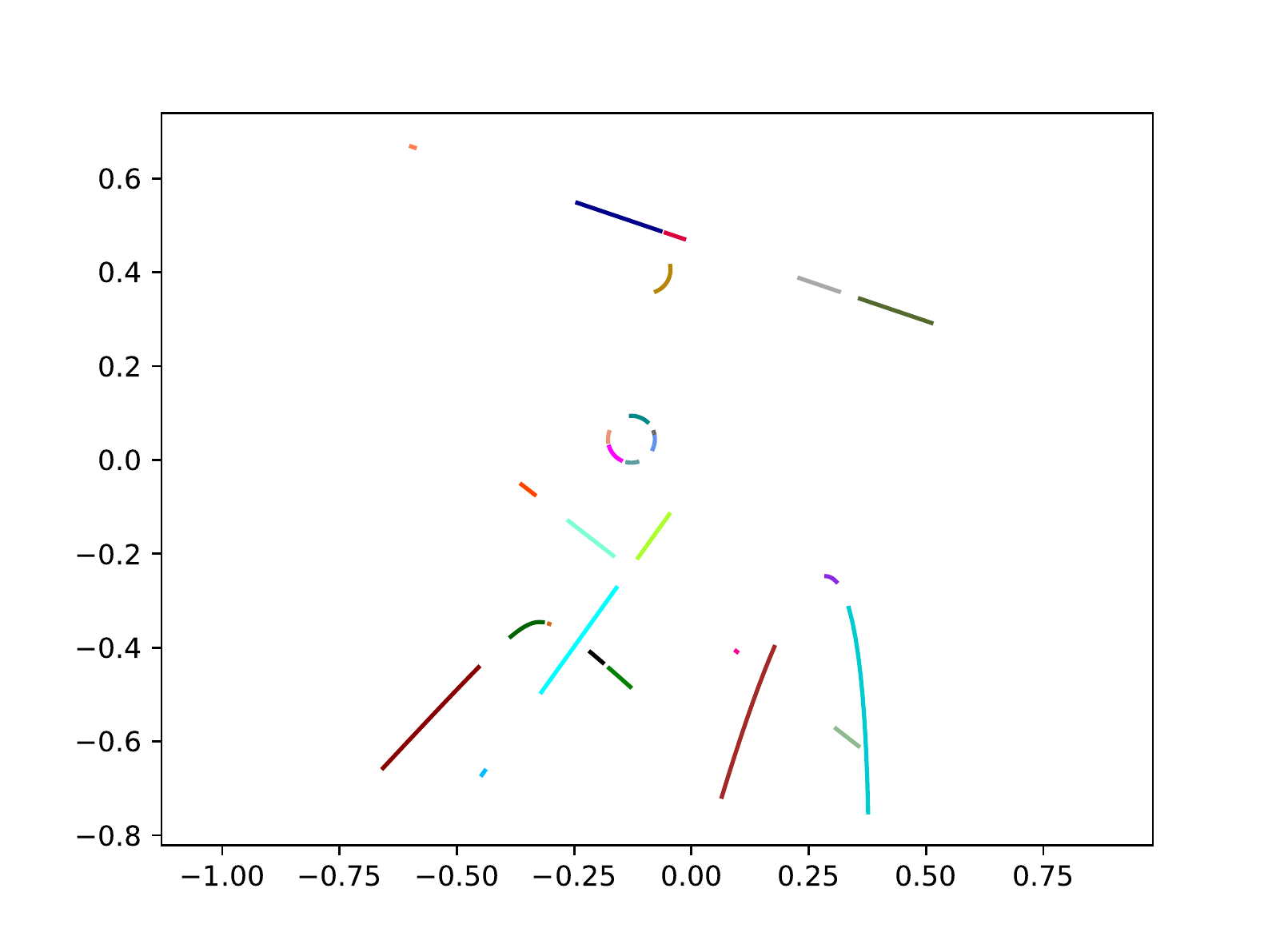}
\end{minipage}
\ \
\begin{minipage}[b]{6.5cm}
\centering
\includegraphics[scale=0.36, clip = true, trim = 30 20 0 40]{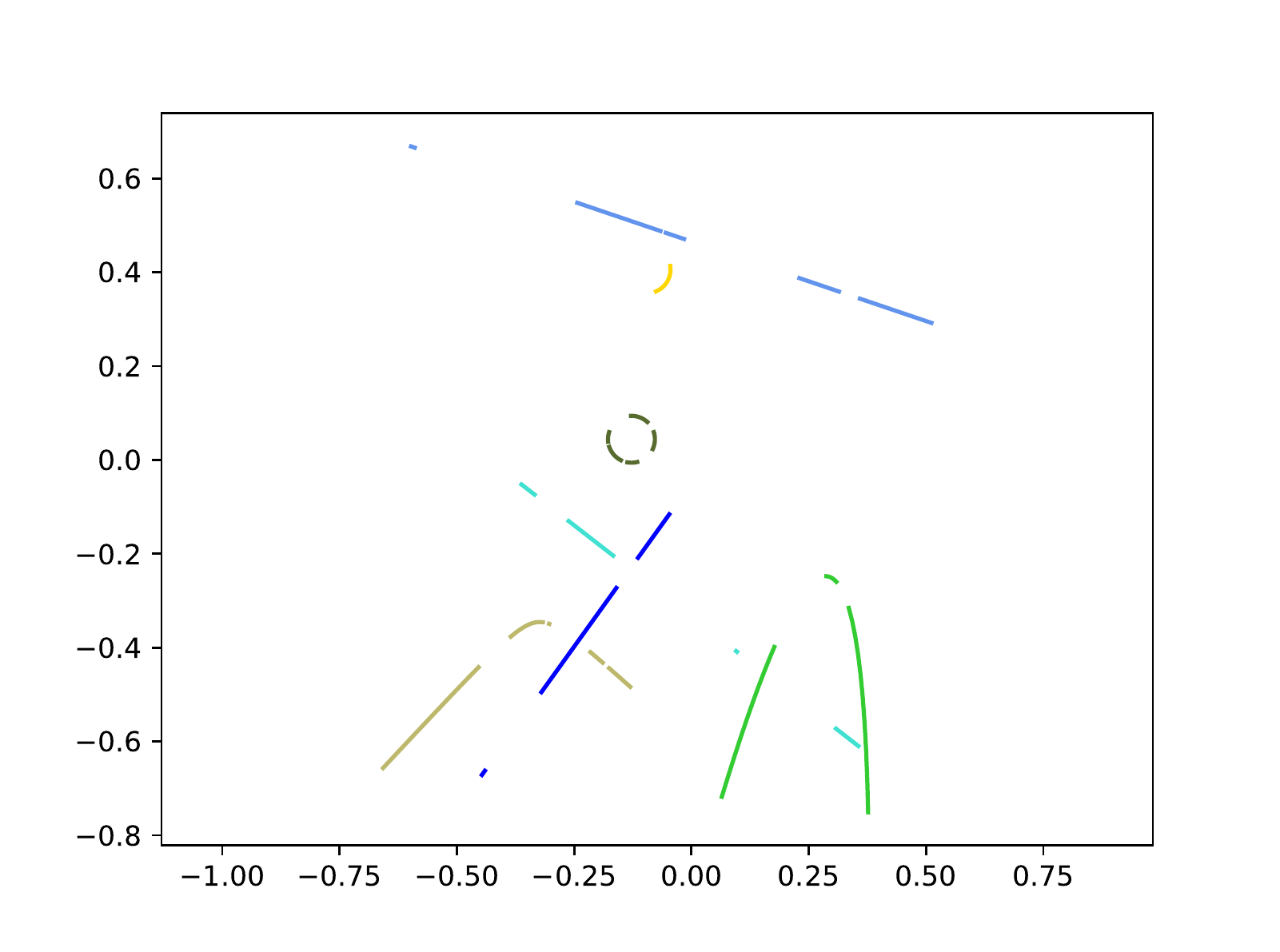}
\end{minipage}
\caption{Initial (left) and corresponding classified dataset (right), with segments colored by cluster.}
\label{fig:Ex1}
\end{figure}

We consider the stopping tolerance presented in Algorithm \ref{alg:StoppCriteta}, as it has experimentally shown to work well with data affected by the only round-off error. The algorithm is run $10^5$ times to compute the average misclassification rate of the approach. This index is given by the ratio between the number of misclassified segments and the total number of segments, averaged over the number of times the algorithm is run. The average misclassification rate is $2.14\cdot10^{-2}$. Equivalently, the $97.86\%$ of the segments are on average correctly classified. Notice that the inaccuracy can be explained by the randomness of the dataset, where some of the dilations can compromise the degree estimation or the stopping criteria (see Section \ref{Theoreticalanalysis} for details about the robustness of discrete approximate implicitization). 

\subsection{Cubic B\'ezier curve approximations to circular segments}
\label{Gear2D}
Let $X$ be the set of segments of the two-dimensional gear in Figure \ref{fig:Ex2}. The gear is built as follows: 
\begin{itemize}
\item Three concentric circles with radii $1, 1.5$ and 2 are approximated by means of cubic B\'ezier curves. Notice that the segments do not lie exactly on the underlying primitives.
\item Concentric line segments are used to connect inner and outer arcs to form the teeth.
\end{itemize}
Again, we apply the algorithm in order to detect the primitive shapes underlying the model. We here use the stopping criteria introduced in Algorithm \ref{alg:StoppCriteta2}, because experimentally more robust when working with polynomial B\'ezier approximations. The resulting output shows the correct detection of the clusters.

\begin{figure}[htp]
\begin{minipage}[b]{6.2cm}
\centering
\includegraphics[scale=0.33, clip = true, trim = 0 20 0 40]{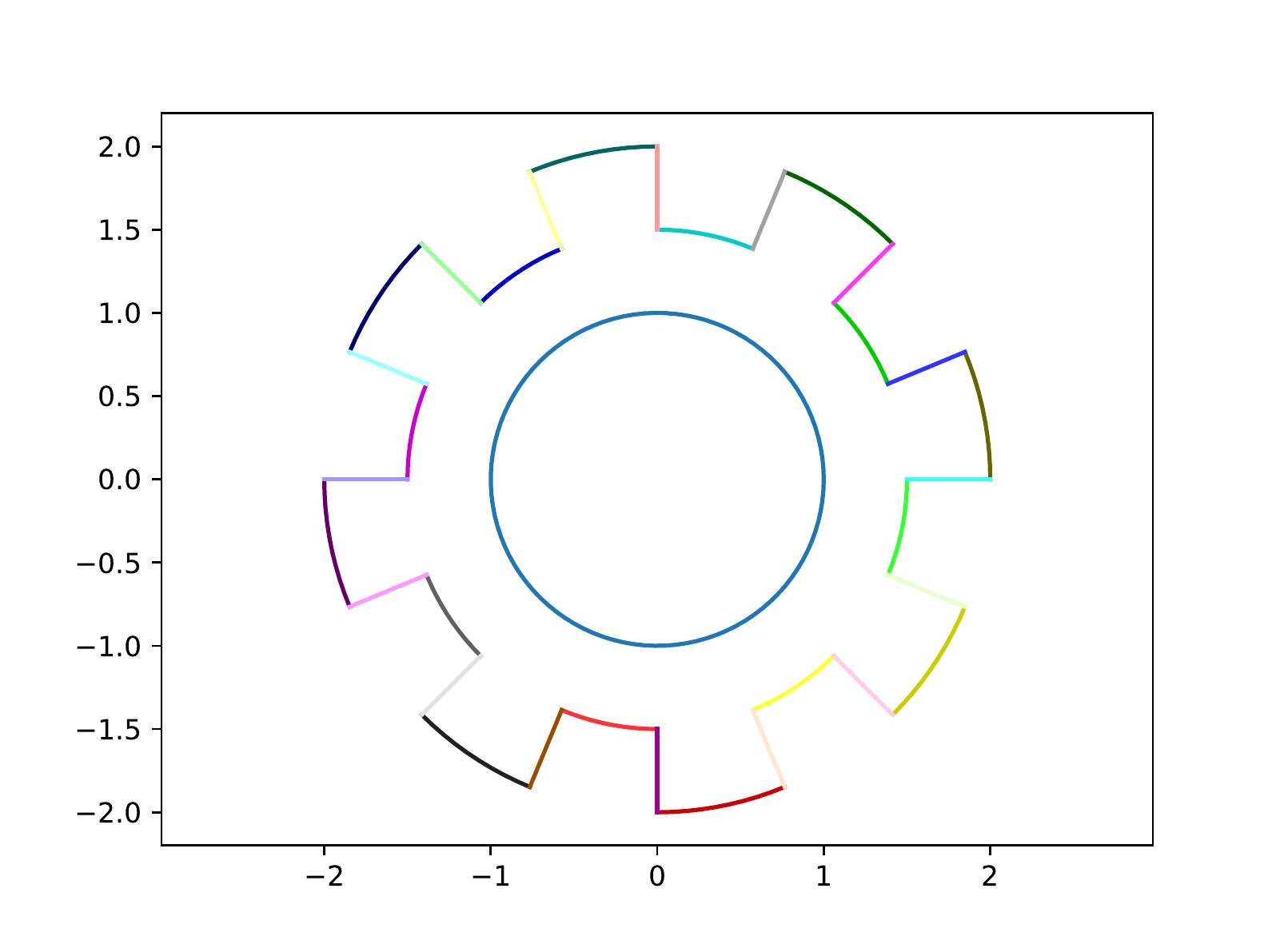} 
\end{minipage}
\ \
\begin{minipage}[b]{6.2cm}
\centering
\includegraphics[scale=0.33, clip = true, trim = 0 20 0 40]{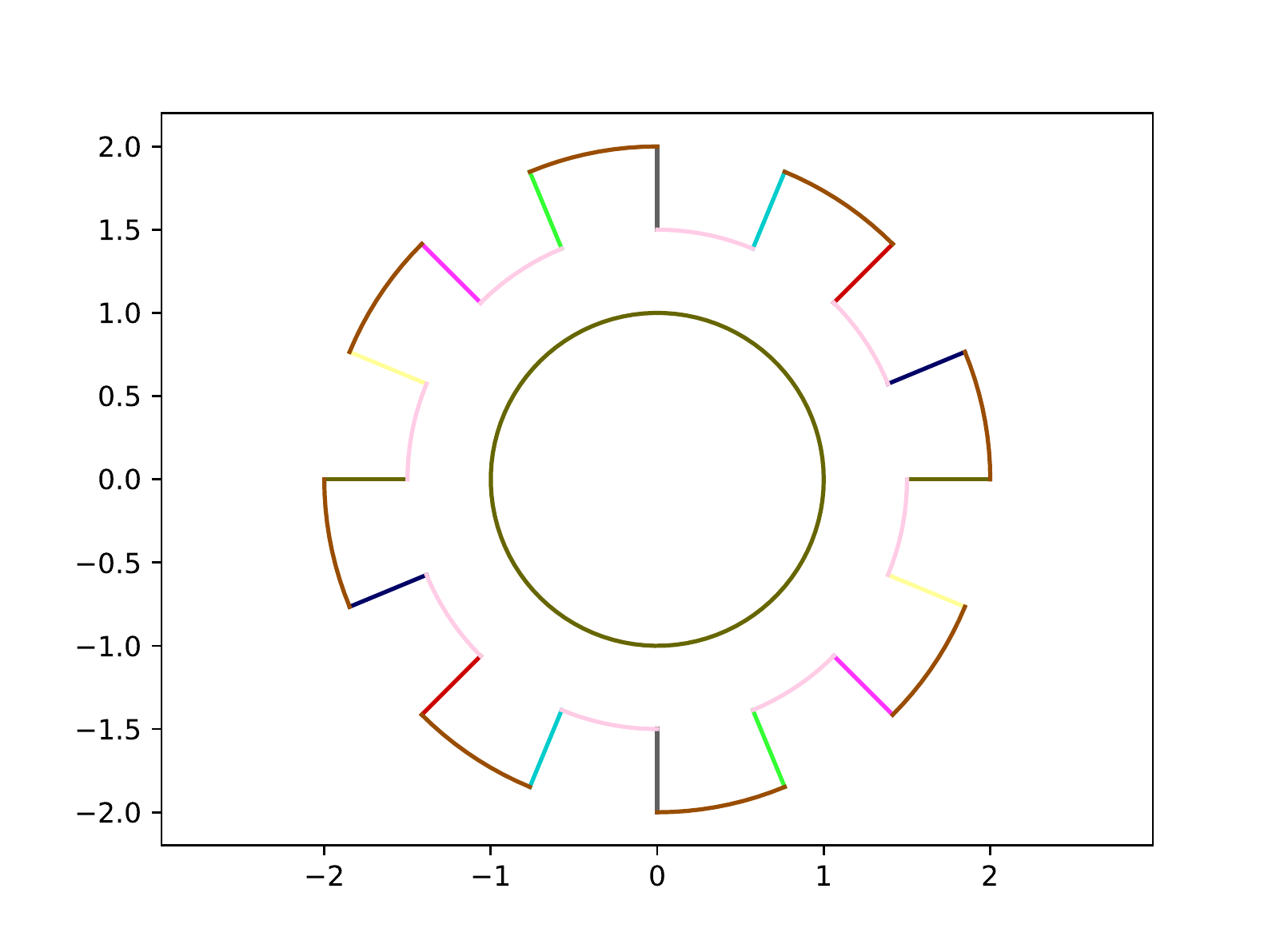}
\end{minipage}
\caption{Initial gear (left) and corresponding classified gear (right), with segments coloured by cluster.} 
\label{fig:Ex2}
\end{figure}

\subsubsection{Increasing number of teeth}
The algorithm is applied to gears with increasing numbers of teeth and run on a 2019 MacBook pro with 2.4 GHz 8-cores Intel\textregistered Core\textsuperscript{TM} i9-processor, resulting in the CPU times shown in Table \ref{table:gear_table}. Although these times suggest that both the dissimilarity matrix assembly and the clustering procedure in the implementation \cite{CodiceMldtt} have $\OOO(P^2)$ complexity, a further breakdown of the CPU times \cite{CodiceMldtt} shows that, in the naive implementation, repeatedly locating the smallest singular value in the clustering procedure yields $\OOO(P^3)$ complexity. However, for $P \le 16384$ patches, this term is dominated by the $\OOO(P^2)$ complexity of the remainder of the algorithm.

    \begin{table}[h!]
        \centering
		\begin{tabular*}{\columnwidth}{@{ }l@{\extracolsep{\stretch{1}}}*{6}{c}@{ }}

           \toprule
& \includegraphics[clip=True, trim = 90 25 90 15,scale=0.14]{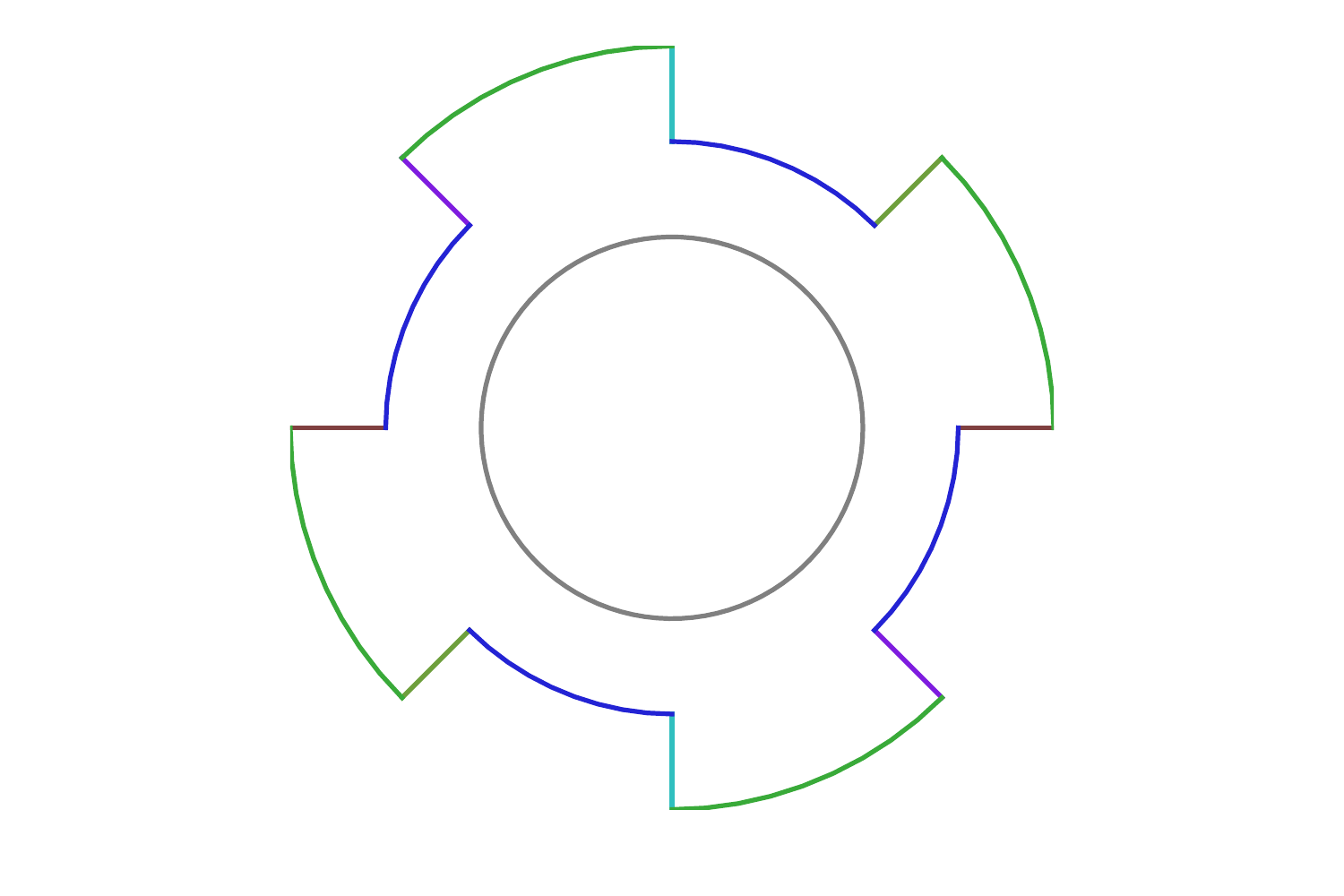}
& \includegraphics[clip=True, trim = 90 25 90 15,scale=0.14]{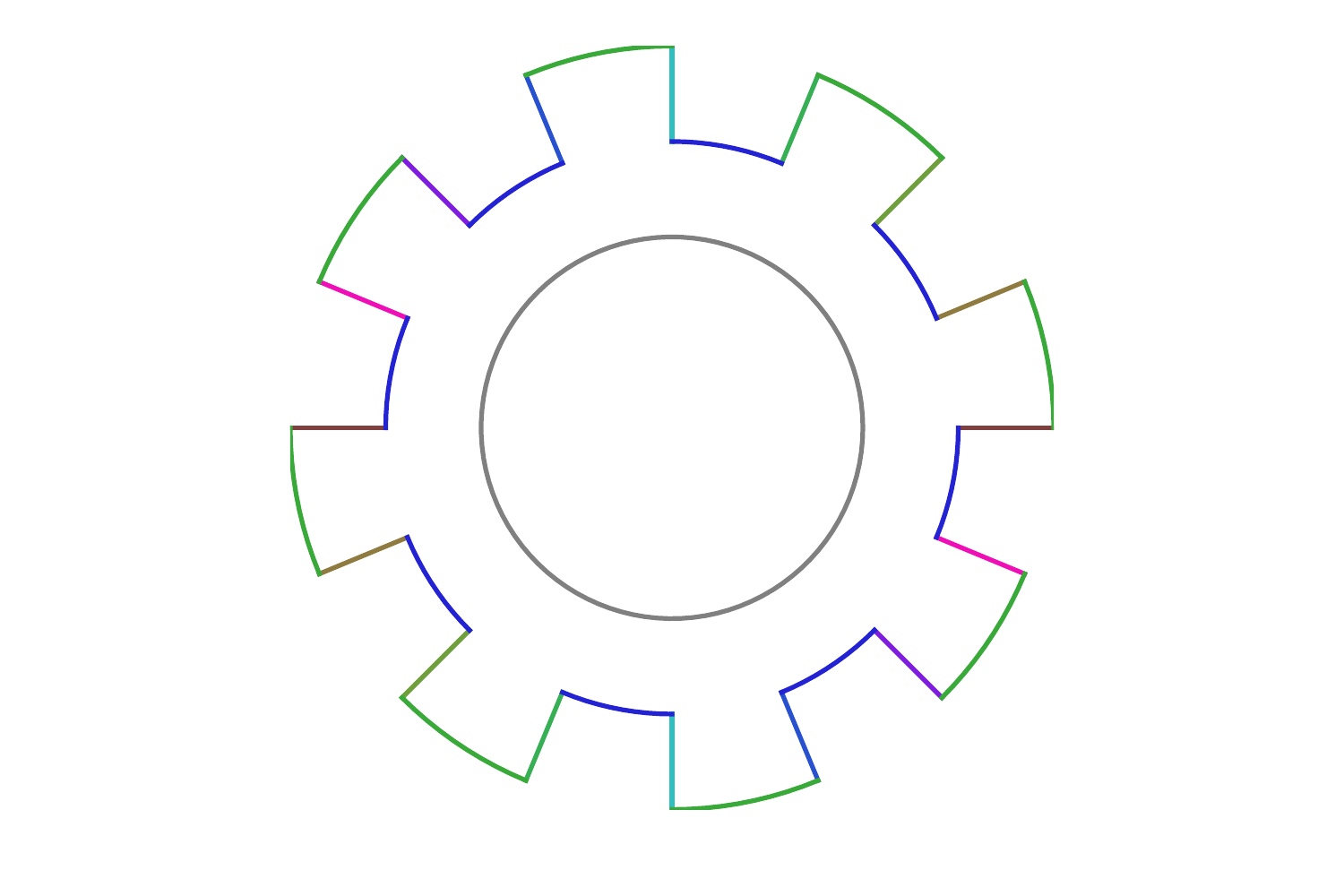}
& \includegraphics[clip=True, trim = 90 25 90 15,scale=0.14]{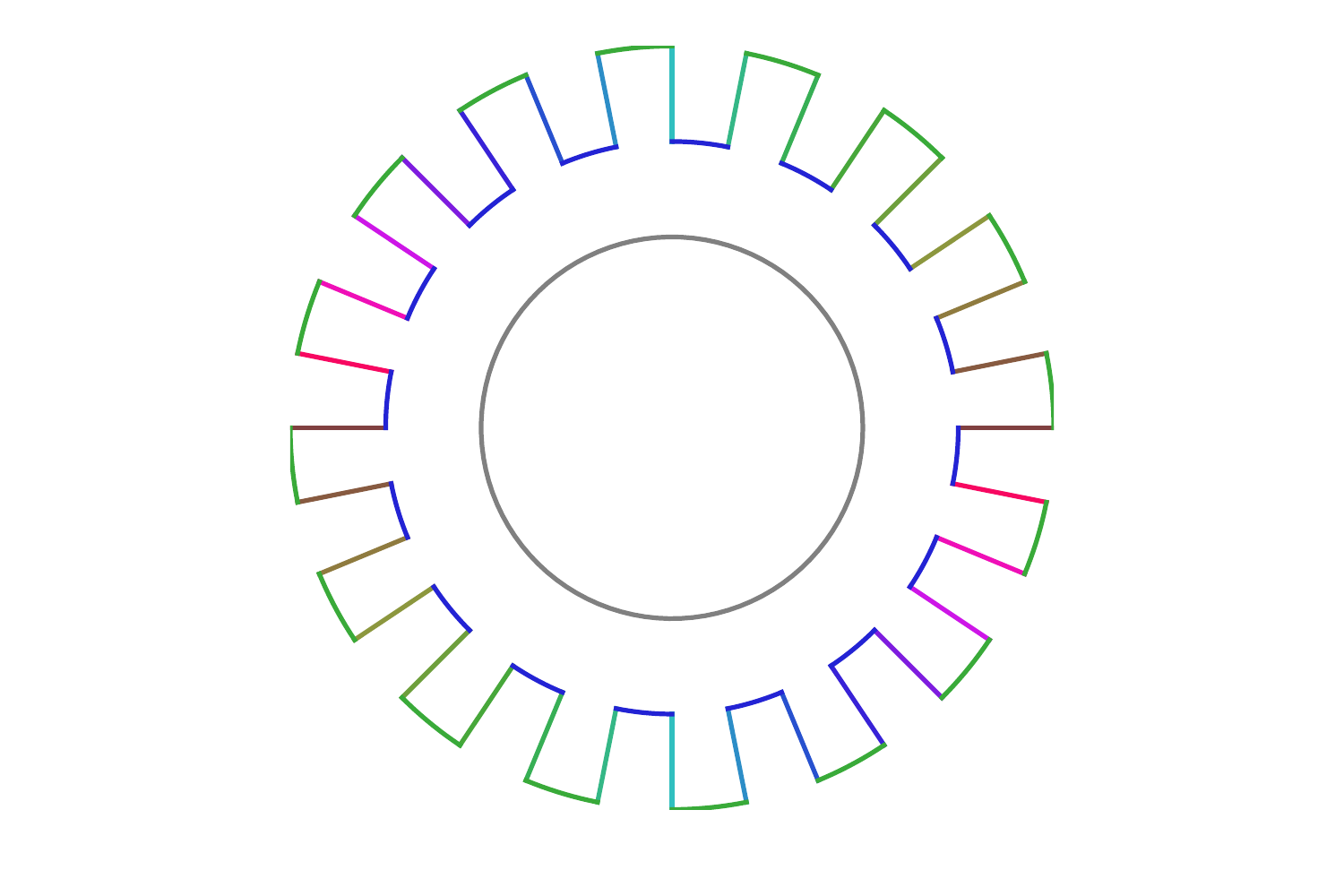}
& \includegraphics[clip=True, trim = 90 25 90 15,scale=0.14]{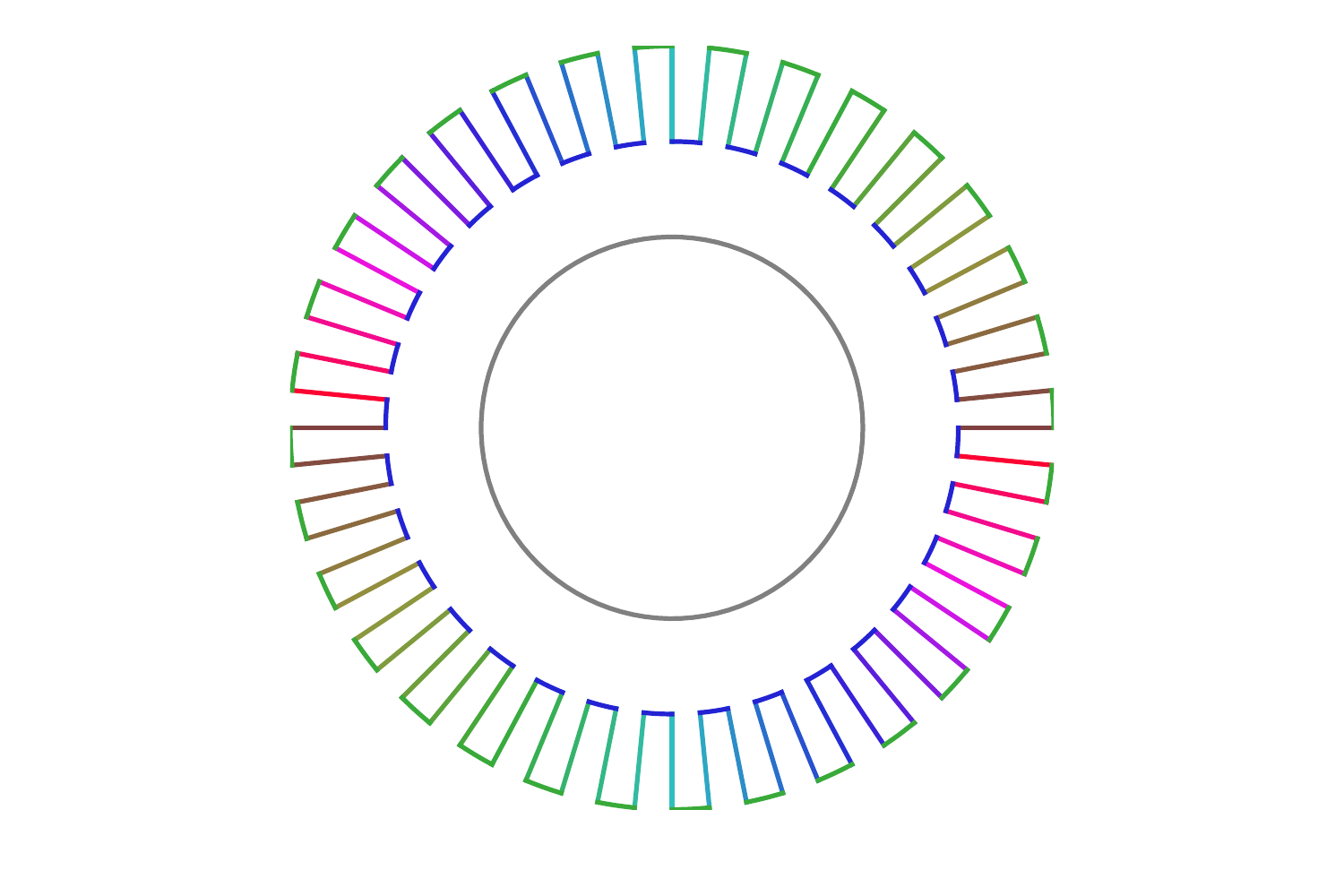}
& \includegraphics[clip=True, trim = 90 25 90 15,scale=0.14]{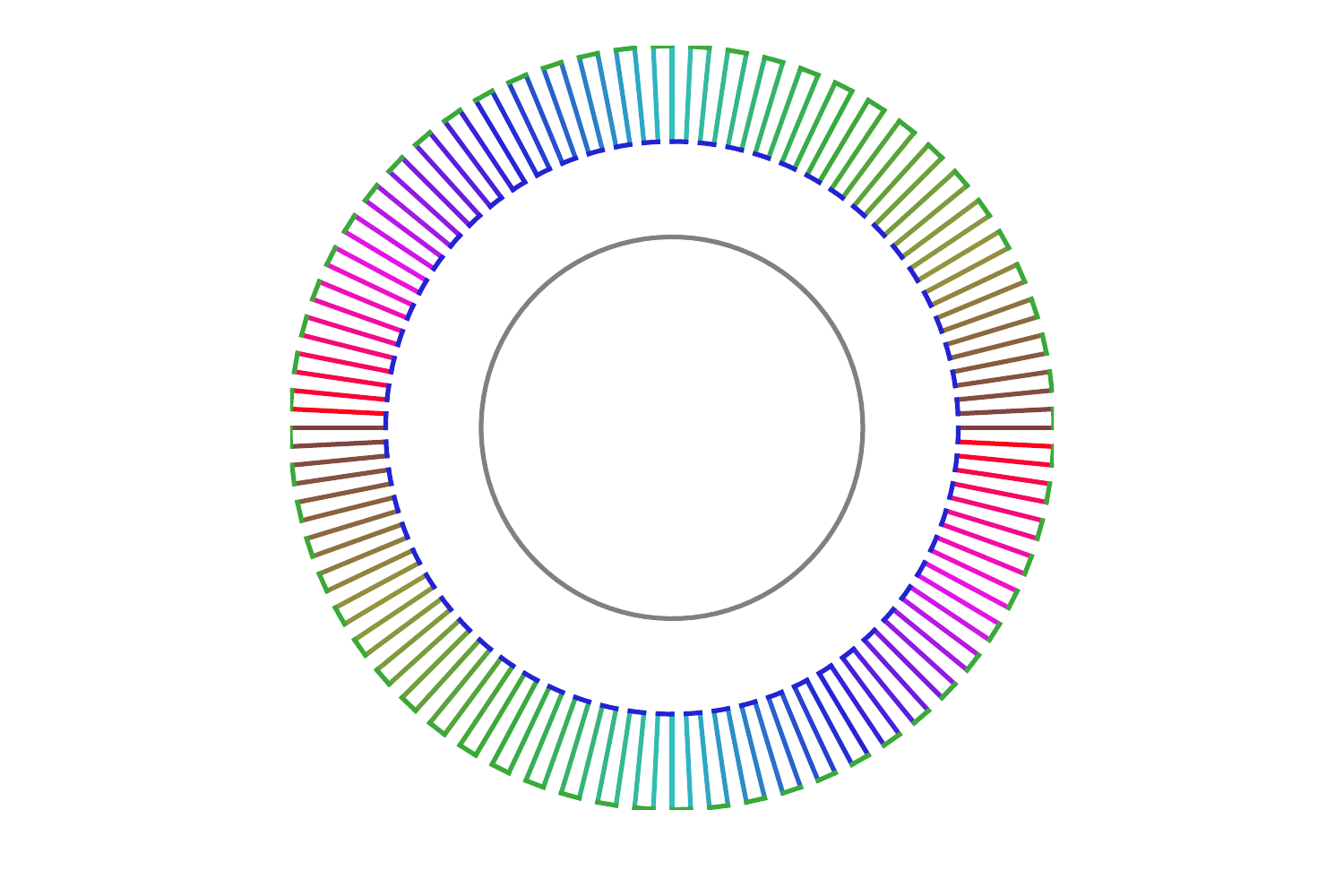}
& \includegraphics[clip=True, trim = 90 25 90 15,scale=0.14]{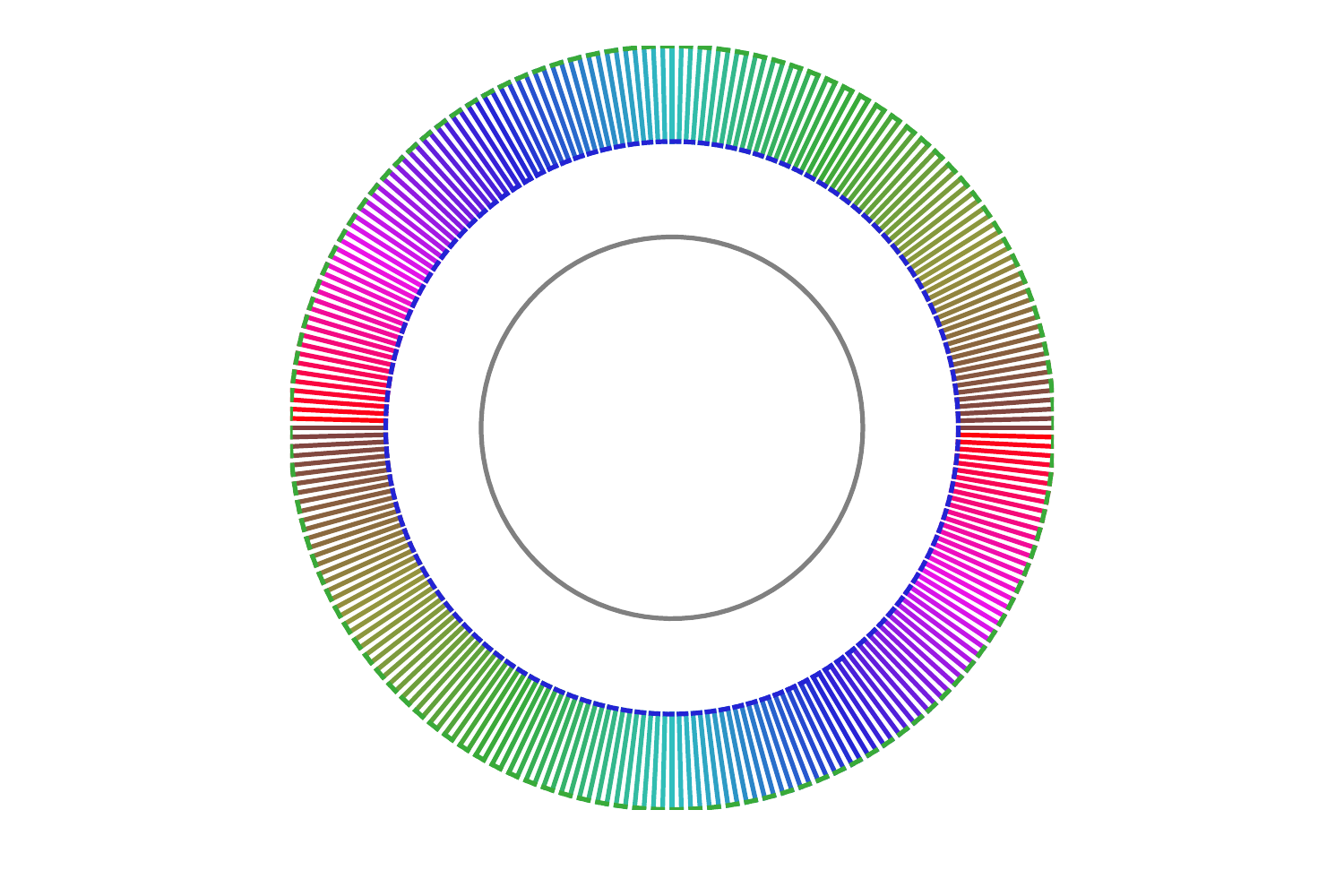}\\
            \midrule
            \# teeth    & 4 & 8 & 16 & 32 & 64 & 128 \\
            \# segments & 17 & 33 & 65 & 129 & 257 & 513 \\
            \midrule
            $t^{\text{assembly}}$  & 0.006 & 0.021 & 0.071 & 0.256 & 1.095 & 4.025\\
            $o^{\text{assembly}}$ & - & 1.708 & 1.778 & 1.843 & 2.096 & 1.879  \\
            \midrule
            $t^{\text{clustering}}$ & 0.002 & 0.003 & 0.008 & 0.042 & 0.115 & 0.405  \\
            $o^{\text{clustering}}$ & - & 0.980 & 1.339 & 2.352 & 1.436 & 1.823 \\
            \midrule
            $t^{\text{total}}$ & 0.009 & 0.025& 0.081 & 0.300 & 1.211 & 4.433\\
            $o^{\text{total}}$   &   -  &  1.483 & 1.681 & 1.888 & 2.015 & 1.872\\
            \bottomrule
        \end{tabular*}
        \caption{A breakdown of the CPU times $t_i$ (seconds) and complexity order $o_i:=\log_2(t_{i+1}/t_{i})$ when applied to gears with exponentially increasing number of teeth/segments.}
    \label{table:gear_table}
    \end{table}

\subsubsection{Increasing Gaussian noise}
Table \ref{table:noisy_gear_table} displays the results of running the algorithm on an 8-tooth gear with addition of synthetic Gaussian noise. Here, the mean is set to $0$ while the standard deviation is increased until the algorithm fails. This experiment suggests that, although well-suited for parametrically defined patches affected by round-off error, this method may require additional information when the input is perturbed by a significant amount of noise.

   \begin{table}[h!]
        \centering
		\begin{tabular*}{\columnwidth}{@{ }l@{\extracolsep{\stretch{1}}}*{4}{c}@{ }}

           \toprule
& \includegraphics[clip=True, trim = 90 25 90 15,scale=0.25]{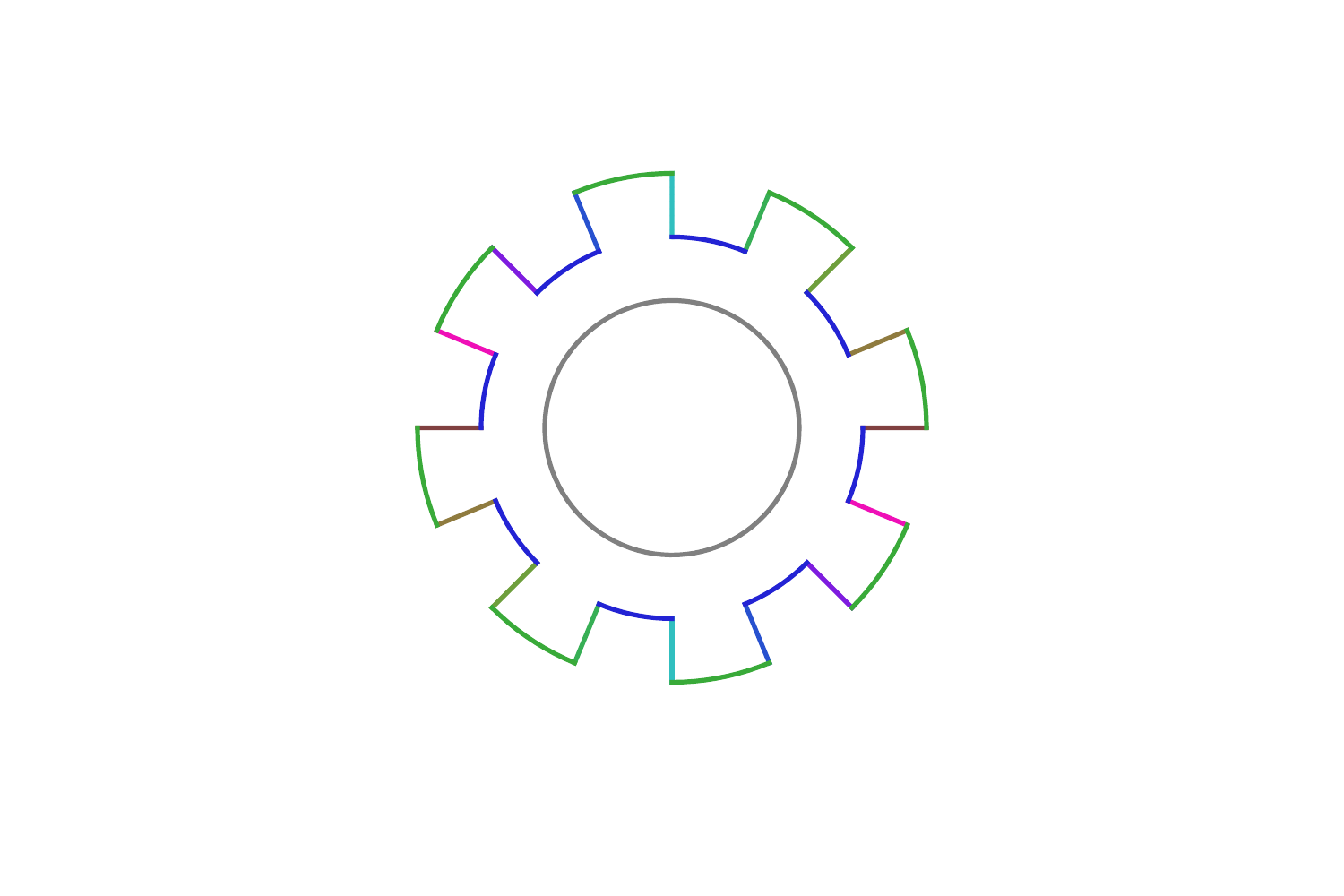}
& \includegraphics[clip=True, trim = 90 25 90 15,scale=0.25]{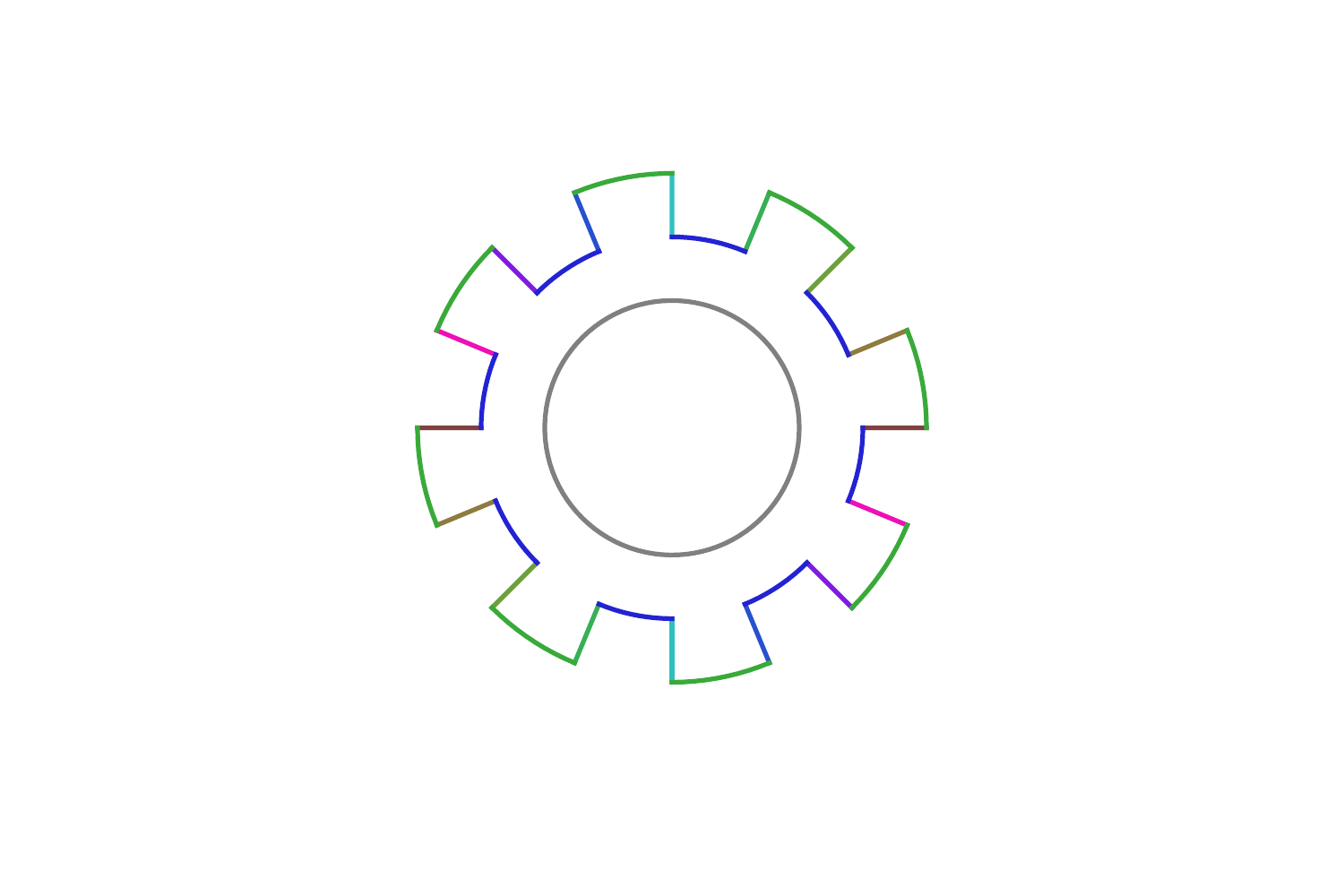}
& \includegraphics[clip=True, trim = 90 25 90 15,scale=0.25]{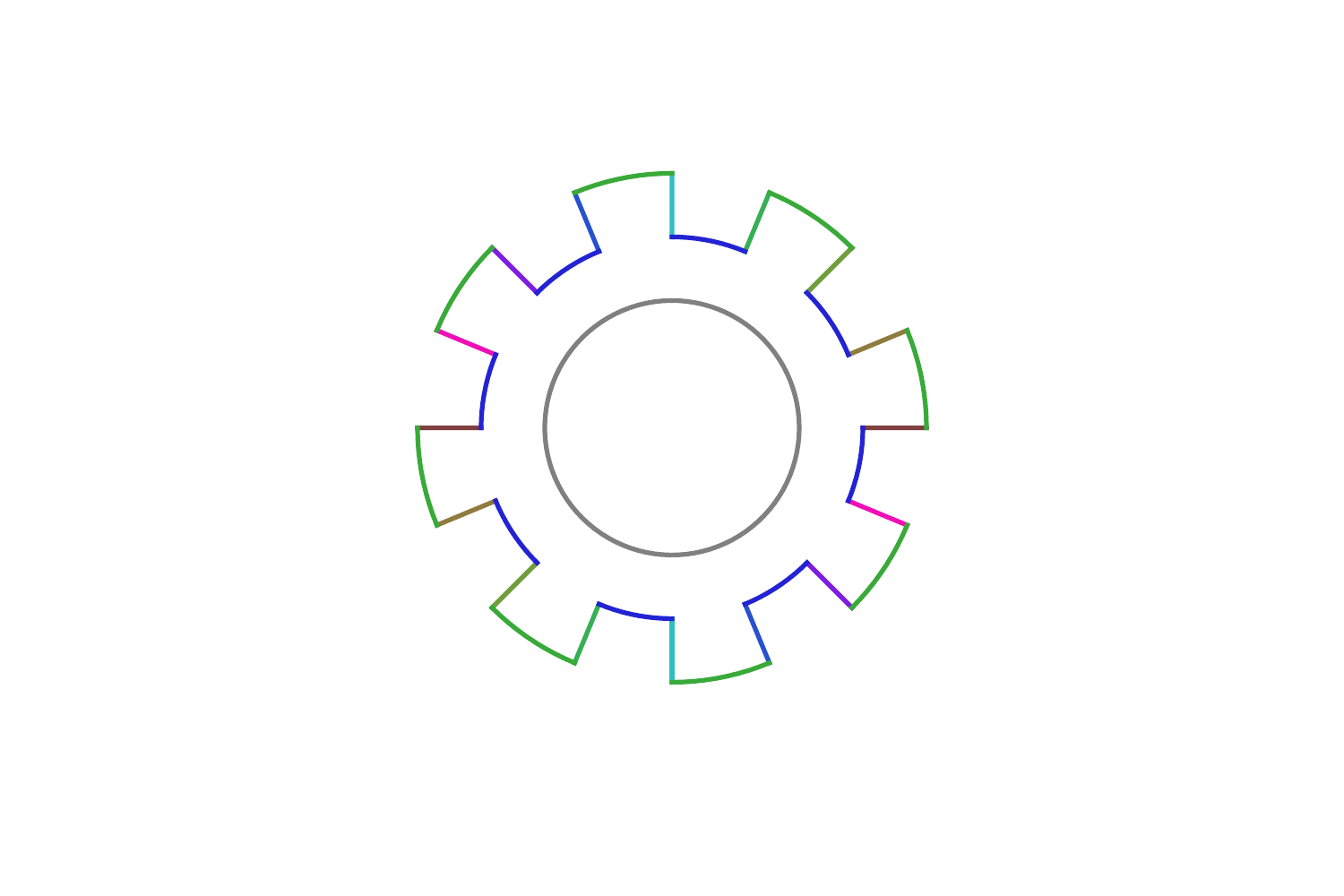}
& \includegraphics[clip=True, trim = 90 25 90 15,scale=0.25]{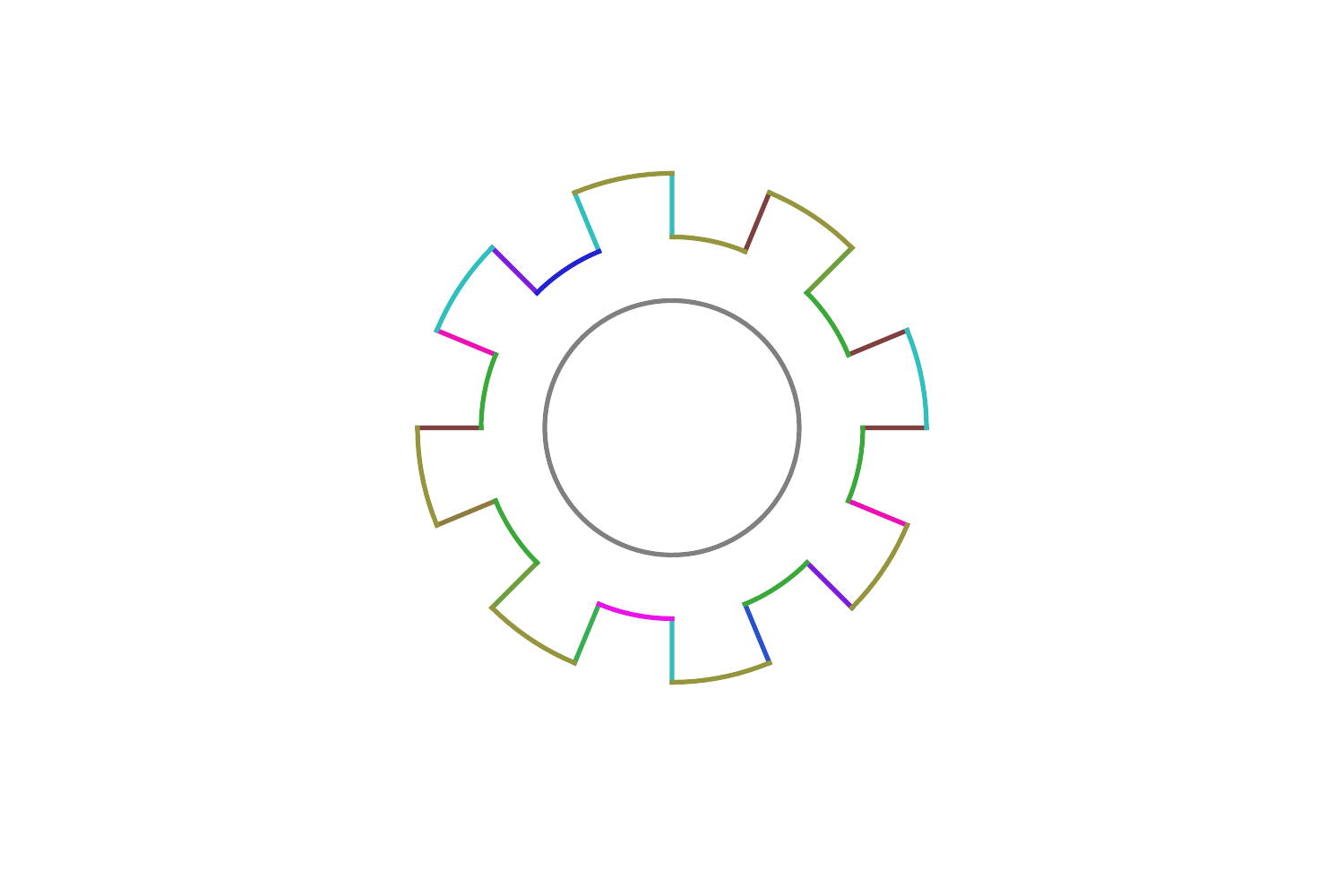}\\
            \midrule
            \ $\sigma$    & $10^{-5}$ & $10^{-4}$ & $10^{-3}$ & $10^{-2}$\\
            \ result & \cmark & \cmark${}^\ast$ & \cmark${}^\ast$ & \xmark\\
            \bottomrule
        \end{tabular*}
        \caption{Sensitivity of the method when an 8-tooth gear is perturbed by Gaussian noise of fixed mean $\mu=0$ and increasing standard deviation $\sigma$ in each direction. Here, \cmark{} and \xmark{} mean, respectively, a correct and an incorrect clustering of the patches. The asterisk signifies that the stopping criterion of Algorithm \ref{alg:StoppCriteta2} fails, but the method yields the correct result via a user-defined threshold (or when the number of clusters is known). }
    \label{table:noisy_gear_table}
    \end{table}
    
\subsection{Surface patches from a real-world industrial example}
In this example we utilize industrial data provided by the high-tech engineering firm STAM, based in Genoa, Italy.
The data, from STAM's Nugear model, consists of 133 trimmed B-spline patches where the underlying geometry comes from planes, cylinders and cones.
The individual patches are coloured randomly in Figure \ref{fig:stamgear} (left).
In order to deal with the trimmed patches in this example, we opt for a random oversampling approach, where we extract a random subsample of 64 points from oversampled surface data, c.f. Remark \ref{rmk:TrimmedSurfaces}. 
This approach is robust as we can always increase the density of the oversampling until we obtain 64 points within the trimmed region. Moreover, 64 points is always sufficient for dealing with patches of algebraic degree less than or equal to two, which are the targeted primitives of this example. 
The samples are generated using the same tessellation code that is used for visualization purposes.

The algorithm for surfaces in $\mathbb{R}^3$ proceeds in exactly the same way as the version for curves in $\mathbb{R}^2,$ with the exception that we must apply approximate implicitization to surfaces in $\mathbb{R}^3$ rather than curves in $\mathbb{R}^2.$
The algorithm correctly classifies all surfaces that should be classified together. The rest of the patches, dark-colored in Figure \ref{fig:stamgear} (right), should all be clustered individually. However, an empirical user-defined threshold  was required in order to avoid some of these surfaces being classified together incorrectly. The issue arises in that some of the bicubic spline surfaces are very close to planar, but not exactly. These surfaces are each classified as quadrics in the first step of the algorithm, but due to the almost linear nature of the surfaces, pairs of them also fit well to quadric surfaces. Thus, it is difficult to choose a threshold that stops the algorithm at the appropriate point.

\begin{figure}[h]
\begin{minipage}[b]{6.1cm}
\centering
\includegraphics[scale=0.40]{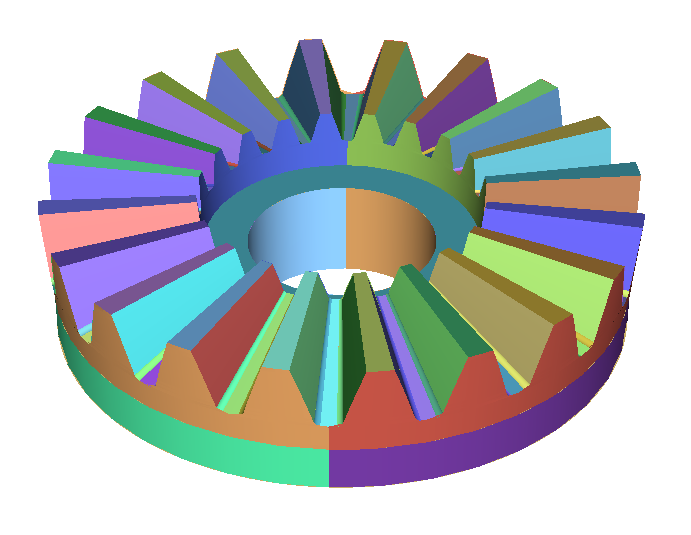}
\end{minipage}
\ \
\begin{minipage}[b]{6.1cm}
\centering
\includegraphics[scale=0.40]{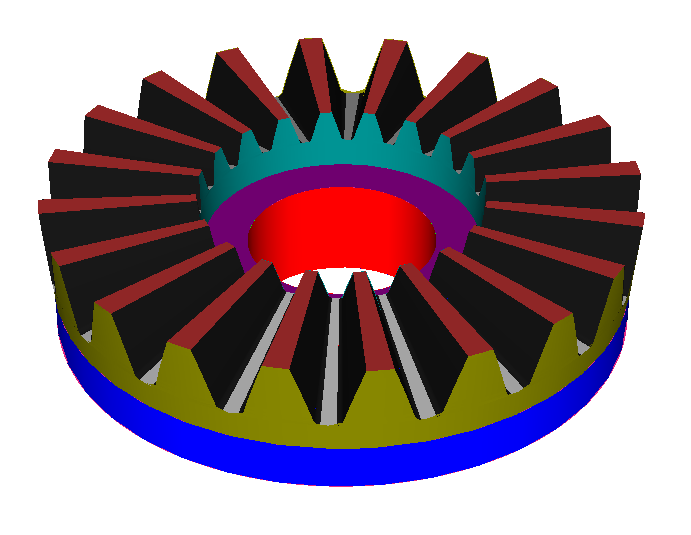}
\end{minipage}
\caption{Initial unclassified (left) and corresponding classified (right) Nugear model provided by STAM, with patches coloured by cluster.}
\label{fig:stamgear}
\end{figure}

\section{Conclusion}\label{sec:Conclusion}
This paper presents a novel method for clustering patches of a CAD model with respect to the underlying surface they originate from, based on approximate implicitization.
The method has been tested on both synthetic and real world industrial data, and some theoretical results are presented to show the properties of the method.
The results show that the approach is computationally feasible and can handle thousands of patches in a matter of seconds.

Our core idea is conceptually quite simple: reverse engineering of CAD patches into their underlying primitives, by using their low-degree implicit representation consistently throughout the algorithm as a basis for establishing similarity in a clustering procedure. However, several modifications were necessary to make this work in practice. Clustering proceeds on local patches, and the goodness of fit of approximate implicitization depends on the scale and position of the coordinate system (e.g., locally a circle looks like a line). This was addressed by introducing an initial tuning step to adapt tolerances to the provided data, and by adding a term to the dissimilarity metric measuring the distance between patches. The latter change also makes our dissimilarity satisfy the formal definition of a dissimilarity for the types of patches considered in this paper.

For future work, we intend to extend the method to other contexts.
For example, reverse engineering of physical models could also be done in this way, although work would be required on stabilizing the algorithm in the case of noisy data.
The algorithm could also be applied to tessellated models, both for the purposes of redesign and upsampling of the tessellation resolution.
In that case, the challenge would be segmenting which parts of the tessellation belong to different `patches' of the model. Finally, we aim at designing a more robust version that can be used to treat data affected by noise and outliers.

%%%%%%%%%%%%%%%%
%%%%%APPENDIX%%%%%%
%%%%%%%%%%%%%%%%
\appendix
\section{Appendix}
\label{Appendix}
In this appendix we provide the proofs left out of Section \ref{Theoreticalanalysis}, for discrete approximate implicitization of degree $d = 1, 2$ to point clouds with sufficiently many points sampled from curves of at least this degree. In this case the collocation matrix $\mathbf{D}$ has full rank.

%%%Scaling%%%
\begin{proof}[Proof of Proposition \ref{prop:scaling}]
Equipping the basis with the lexicographic order, we can express the collocation matrix of the scaled curve $\mathbf{p}_{a_1,a_2}$ as 
 \[\mathbf{D}_{a_1,a_2}=\mathbf{D}_{1,1}\cdot\mathbf{S}_{a_1,a_2},\qquad 
 \mathbf{S}_{a_1,a_2}:=\text{diag}([1,a_1,a_2,a_1^2,a_1a_2,a_2^2,\dots]),
 \]
with $\mathbf{D}_{1,1}:=\mathbf{D}$. Notice that: 
\begin{itemize}
\item The smallest singular value $\sigma_{\min}^{(m)}\left(\mathscr{P}_{\mathbf{a}}\right)$ of $\mathbf{D}_{a_1,a_2}$ is the reciprocal of the largest singular value of
\[ \mathbf{D}_{a_1,a_2}^{\dagger}=\mathbf{S}_{a_1,a_2}^{-1}\cdot \mathbf{D}_{1,1}^{\dagger}, \]
where $\dagger$ denotes the Moore-Penrose inverse.
\item The largest singular value of $\mathbf{D}_{a_1,a_2}^{\dagger}$ is the square root of the largest eigenvalue of $\mathbf{D}_{a_1,a_2}^{\dagger}(\mathbf{D}_{a_1,a_2}^{\dagger})^T$.
\end{itemize}
Therefore 
\[\sigma_{\min}^{(m)}\left(\mathscr{P}_{\mathbf{a}}\right)=\Bigl[\lambda_{\text{max}}\left(\mathbf{B}\right)\Bigr]^{-1/2},\]
where
\[\mathbf{B}=
\mathbf{S}_{a_1,a_2}^{-1}\cdot
\mathbf{C}\cdot
\mathbf{S}_{a_1,a_2}^{-1}, \qquad
\mathbf{C}:=\mathbf{D}_{1,1}^{\dagger}\cdot{\mathbf{D}_{1,1}^{\dagger}}^T.\]

Suppose $\mathbf{B}$ has eigenvalues $\lambda_1\ge\lambda_2\ge\dots\ge\lambda_n\ge0$. Then
\[
tr\left(\mathbf{B}\right)=\sum_{i=1}^n\lambda_i=c_{11}+c_{22}\dfrac{1}{a_1^2}+c_{33}\dfrac{1}{a_2^2}+\dots
\]
where $c_{i,i}$ is the square of the Euclidean norm of the $i$-th row of $\mathbf{D}_{a_1,a_2}^{\dagger}$. Since $\mathbf{D}_{a_1,a_2}^{\dagger}$ has full rank whenever $a_1,a_2\neq 0$, then $tr\left(\mathbf{B}\right)$ (and consequently $\lambda_1$) approaches $+\infty$ when $a_1$ or $a_2$ approaches zero. It follows that $\sigma_{\min}^{(m)}\left(\mathscr{P}_{\mathbf{a}}\right)$ approaches zero when $a_1$ or $a_2$ approaches zero.
\end{proof}

%%%Translation%%%
\begin{proof}[Proof of Proposition \ref{prop:translation}]
The collocation matrix $\mathbf{D}_{a_1,a_2}$ has the form
\[
\mathbf{D}_{a_1,a_2}=\mathbf{D}_{0,0}\mathbf{T}_{a_1,a_2},\qquad
\mathbf{T}_{a_1,a_2}:=
\left(
\begin{array}{c|cc|ccc}
1 & a_1 & a_2 & a_1^2 & a_1a_2 & a_2^2\\ 
0 & 1 & 0 & 2a_1 & a_2 & 0 \\ 
0 & 0 & 1 & 0 & a_1 & 2a_2 \\
0& 0 & 0 & 1& 0 & 0 \\
0& 0 & 0 & 0& 1 & 0 \\
0& 0 & 0 & 0& 0 & 1
\end{array}
\right),
\]
and $\mathbf{D}_{0,0}:=\mathbf{D}$. Notice that, since discrete approximate implicitization of degree~$2$ is applied to a curve of implicit degree greater than $2$, we can assume that the collocation matrix has full rank. Notice also that: 
\begin{itemize}
\item The smallest singular value of $\mathbf{D}_{a_1,a_2}$ is the reciprocal of the largest singular value of \[\mathbf{D}_{a_1,a_2}^{\dagger}=\mathbf{T}_{-a_1,-a_2}\cdot \mathbf{D}_{0,0}^{\dagger}.\]
\item The largest singular value of $\mathbf{D}_{a_1,a_2}^{\dagger}$ is the square root of the largest eigenvalue of $\mathbf{B} := \mathbf{D}_{a_1,a_2}^{\dagger}(\mathbf{D}_{a_1,a_2}^{\dagger})^T$.
\end{itemize}
Therefore
\[\sigma_{\min}^{(m)}\left(\mathscr{P}_{\mathbf{a}}\right)=\Bigl[\lambda_{\max}\left(\mathbf{B}\right)\Bigr]^{-1/2},\]
where
\[
\mathbf{B}:=\mathbf{T}_{-a_1,-a_2}
\mathbf{C}
\mathbf{T}_{-a_1,-a_2}^T, \qquad
\mathbf{C} := \mathbf{D}_{0,0}^{\dagger}\left(\mathbf{D}_{0,0}^{\dagger}\right)^T.\]

Suppose $\mathbf{B}$ has eigenvalues $\lambda_1\ge\lambda_2\ge\dots\ge\lambda_n\ge0$. Then
\begin{equation}
\label{eqn:sumeigen1}
tr\left(\mathbf{B}\right)=\sum_{i=1}^n\lambda_i=\sum_{0\le i+j\le4}\alpha_{i,j}a_1^ia_2^j,
\end{equation}
for some $\alpha_{i,j}\in\mathbb{R}$. Notice that $\alpha_{0,4}>0$, since it is the square of the Euclidean norm of the 4-th row of the full-rank matrix $\mathbf{D}_{0,0}^{\dagger}$. A similar argument holds for $\alpha_{4,0}$. It follow that $tr\left(\mathbf{B}\right)$ (and consequently $\lambda_1$) approaches $+\infty$ when either $a_1$ or $a_2$ approaches $\infty$, and consequently $\sigma_{\min}^{(m)}\left(\mathscr{P}_{\mathbf{a}}\right)$ approaches zero when either $a_1$ or $a_2$ approaches $\infty$.
\end{proof}

%%%Rotation%%%
\begin{proof}[Proof of Proposition \ref{prop:rotation}]
We treat separately the following cases:

\textit{Discrete approximate implicitization of degree $1$.} 
The collocation matrix $\mathbf{D}_{\theta}$ can be expressed as 
\[
\mathbf{D}_{\theta}=\mathbf{D}_0\cdot \mathbf{R}_{\theta},\qquad
\mathbf{R}_{\theta}:=
\begin{pmatrix}1 & 0 & 0 \\ 0 & \cos{\theta} & \sin{\theta}\\ 0 & - \sin{\theta} & \cos{\theta} \end{pmatrix},
\]
with $\mathbf{D}_0:=\mathbf{D}$. It follows that the SVD of $\mathbf{D}_{\theta}$ is
\[
\mathbf{D}_{\theta}=\mathbf{D}_0\mathbf{R}_{\theta}=\mathbf{U}\boldsymbol{\Sigma}\mathbf{V}^T\mathbf{R}_{\theta}=
\mathbf{U}\boldsymbol{\Sigma}\widetilde{\mathbf{V}}^T.
\]
where $\widetilde{\mathbf{V}}:=\mathbf{R}_{\theta}^T\mathbf{V}$ is unitary since $\mathbf{R}_{\theta}$ and $\mathbf{V}$ are unitary. Thus, the smallest singular value of discrete approximate implicitization of degree~$1$ is not influenced by rotations, and it suffices to choose $\alpha=\beta=\sigma_{\min}^{(m)}\left(\mathscr{P}_{0}\right)$.

\textit{Discrete approximate implicitization of degree greater than $1$.}
The collocation matrix $\mathbf{D}_{\theta}$ can be expressed as 
\[
\mathbf{D}_{\theta}=\mathbf{D}_0\cdot\mathbf{R}_{\theta},\]
where
$\mathbf{R}_{\theta}$ is the block diagonal matrix
\[
\mathbf{R}_{\theta}=
\left(
\begin{array}{c|cc|ccc|c}
	1 & 0 & 0 & 0 & 0 & 0 & \cdots \\ 
	\hline
	0 & \cos{\theta} & \sin{\theta} & 0 & 0 & 0 & \cdots\\ 
	0 & - \sin{\theta} & \cos{\theta} & 0 & 0 & 0 & \cdots \\
	\hline
	0 & 0 & 0 & \cos^2{\theta} & \sin{\theta}\cos{\theta} & \sin^2{\theta} & \cdots\\
	0 & 0 & 0 & -2\sin{\theta}\cos{\theta} & \cos^2{\theta}-\sin^2{\theta} & 2\sin{\theta}\cos{\theta} & \cdots\\
	0 & 0 & 0 & \sin^2{\theta} & -\sin{\theta}\cos{\theta} & \cos^2{\theta} & \cdots\\
	\hline
	\vdots & \vdots &\vdots &\vdots &\vdots &\vdots & \ddots
\end{array}
\right)
\]
The case of degree greater than $1$ differs since the matrix $\mathbf{R}_{\theta}$ is not unitary. The smallest singular value of $\mathbf{D}_{\theta}$ is the square root of the smallest eigenvalue of $\mathbf{D}_{\theta}^T\mathbf{D}_{\theta}$. The eigenvalues of $\mathbf{D}_{\theta}^T\mathbf{D}_{\theta}$ are continuous functions in $\theta$. Since $\theta\in[0,2\pi]$ lies in a compact space, we conclude by the extreme value theorem that each of these eigenvalues has a compact and connected range. Since $\mathbf{D}_{\theta}$ has full rank, we conclude that the range does not contain~$0$. $\hfill \qedhere$

\end{proof}

\section*{Acknowledgments}
This work has received funding from the European Union's Horizon 2020 research and innovation programme under grant agreement No 675789 and No 680448.
%BIBLIOGRAFIA
\bibliographystyle{unsrt}
\bibliography{Bibliografia}
\end{document}